\documentclass[a4paper,11pt]{amsart}
\pdfoutput=1

\usepackage{mathtools}
\mathtoolsset{showonlyrefs=true}

\usepackage{amsmath,amssymb,amsthm,amsfonts, mathrsfs, fancyhdr}
\usepackage[usenames,dvipsnames,usenames,dvipsnames,svgnames,table]{xcolor}
\usepackage[expansion=false]{microtype}
\usepackage[all]{xy}

\usepackage[leqno]{amsmath}
\makeatletter
\newcommand{\leqnomode}{\tagsleft@true\let\veqno\@@leqno}
\newcommand{\reqnomode}{\tagsleft@false\let\veqno\@@eqno}
\makeatother

\usepackage{anyfontsize}

\usepackage{upgreek}

\usepackage{enumitem}
\makeatletter
\newcommand{\mylabel}[2]{#2\def\@currentlabel{#2}\label{#1}}
\makeatother

\setlist[description]{leftmargin=*}
\usepackage{amscd}
\usepackage[active]{srcltx}
\usepackage{verbatim}
\usepackage{epsfig,graphicx,color}
\usepackage{graphicx}
\usepackage{mathtools,xparse}
\usepackage[english]{babel}

\usepackage{tabularx,ragged2e}
\newcolumntype{L}{>{\RaggedRight\arraybackslash}X}

\usepackage{pgf, tikz, tikz-cd}
\usepackage{dsfont}
\usetikzlibrary{matrix}

\usepackage{capt-of}
\usepackage{float}

\usepackage{verbatim}
\usepackage{cite}
\usepackage{manyfoot}

\usepackage{palatino,hhline, bbm}

\usepackage[left=2.7cm,right=2.7cm,top=3cm,bottom=3cm]{geometry}

\usepackage[unicode,bookmarks, pdftex, backref = page]{hyperref}
\definecolor{newblue}{RGB}{0,102,204}
\definecolor{newred}{RGB}{206,32,41}
\hypersetup{colorlinks=true,citecolor=newblue,linkcolor=newred,
  urlcolor=magenta,
pdfpagemode=UseNone, breaklinks=true}

\usepackage{thmtools}

\usepackage{apptools}
\AtAppendix{\counterwithin{theorem}{section}}

\newtheorem{theorem}{Theorem}[section]
\newtheorem{proposition}[theorem]{Proposition}
\newtheorem{lemma}[theorem]{Lemma}
\newtheorem{corollary}[theorem]{Corollary}

\newtheorem*{A}{Main Theorem}

\theoremstyle{definition}
\newtheorem{remark}[theorem]{Remark}
\newtheorem{example}[theorem]{Example}
\newtheorem{definition}[theorem]{Definition}

\renewcommand{\setminus}{-}

\newcommand{\mZ}{\mathbb{Z}}

\renewcommand{\to}{\longrightarrow}
\renewcommand{\t}{\mathsf{t}}
\renewcommand{\r}{\mathsf{r}}
\newcommand{\z}{\mathsf{z}}
\newcommand{\x}{\mathsf{x}}
\newcommand{\y}{\mathsf{y}}
\renewcommand{\S}{\mathfrak{S}}
\newcommand{\bt}{\mathsf{\bar{t}}}
\newcommand{\br}{\mathsf{\bar{r}}}
\newcommand{\bz}{\mathsf{\bar{z}}}

\usepackage{epigraph}
\usepackage{listings}
\lstdefinestyle{customc}{
  belowcaptionskip=1\baselineskip,
  breaklines=true,
  xleftmargin=\parindent,
  language=GAP,
  showstringspaces=false,
  basicstyle=\footnotesize\ttfamily,
  keywordstyle=\bfseries\color{green!40!black},
  commentstyle=\itshape\color{purple!40!black},
  identifierstyle=\color{blue},
  stringstyle=\color{olive},
}
\title[Extra-special quotients of surface braid groups and double Kodaira fibrations]{Extra-special quotients of surface braid groups and double Kodaira fibrations with small signature}

\date{}
\author[Francesco Polizzi]{Francesco Polizzi $^{*}$} 
\address{Dipartimento di Matematica e Informatica
  \newline\indent
  Universit\`a della Calabria
  \newline\indent
Ponte Pietro Bucci 30B, I-87036 Arcavacata di Rende, Cosenza, Italy}
\email{polizzi@mat.unical.it}

\author[Pietro Sabatino]{Pietro Sabatino}
\address{Via Val Sillaro 5 \\ 00141 Roma, Italy}
\email{pietrsabat@gmail.com}



\begin{document}


\begin{abstract}
  We study some special systems of generators on finite groups, introduced
  in previous work by the first author and called \emph{diagonal double
  Kodaira structures}, in order to investigate finite non-abelian
  quotients of
  the pure braid group on two strands $\mathsf{P}_2(\Sigma_b)$, where
  $\Sigma_b$ is a closed
  Riemann surface of genus $b$.
  In particular, we prove that, if a finite group $G$ admits a diagonal
  double Kodaira structure, then $|G|\geq 32$, and
  equality holds if and only if $G$ is extra-special. In the last section,
  as a geometrical application of our algebraic results, we construct two
  $3$-dimensional families of double Kodaira fibrations having signature $16$. Such surfaces are different from the ones recently constructed by Lee, L\"{o}nne and Rollenske and, as far as we know, they provide the first examples of positive-dimensional families of double Kodaira fibrations with small signature.
\end{abstract}

\maketitle



\setcounter{section}{-1}

\section{Introduction} \label{sec:intro}
A \emph{Kodaira fibration} is a smooth, connected holomorphic fibration
$f_1 \colon S \to B_1$, where $S$ is a compact complex surface and $B_1$
is a compact closed curve, which is not isotrivial (this means that not all
fibres are biholomorphic each other). The genus $b_1:=g(B_1)$ is called
the \emph{base genus} of the fibration, and the genus $g:=g(F)$, where
$F$ is any fibre, is called the \emph{fibre genus}. A surface $S$ that is
the total space of a Kodaira fibration is called a \emph{Kodaira fibred
surface}. For every Kodaira fibration, we have $b_1 \ge 2$ and $g \geq 3$,
see \cite[Theorem 1.1]{Kas68}. Since the fibration is smooth, the condition
on the base genus implies that $S$ contains no rational or elliptic curves;
hence $S$ is minimal and, by the sub-additivity of the Kodaira dimension,
it is of general type, hence algebraic.

An important topological invariant of a Kodaira fibred surface $S$ is its
\emph{signature} $\sigma(S)$, namely the signature of the intersection
form on the middle cohomology group $H^2(S, \, \mathbb{R})$. Actually, the
first examples of Kodaira fibrations (see \cite{Kod67}) were constructed
in order to show that $\sigma$ is not multiplicative for fibre bundles. In
fact, $\sigma(S)>0$ for every Kodaira fibration (see the introduction
to \cite{LLR17}), whereas $\sigma(B_1)=\sigma(F)=0$, hence $\sigma(S)
\neq \sigma(B_1)\sigma(F)$; by \cite{CHS57}, this in turn means that the
monodromy action of $\pi_1(B)$ on the rational cohomology ring $H^*(S, \,
\mathbb{Q})$ is non-trivial.

Every Kodaira fibred surface $S$ has the structure of a real surface
bundle over a smooth real surface, and so $\sigma(S)$ is divisible
by $4$, see \cite{Mey73}.  If, in addition, $S$ has a spin structure,
i.e. its canonical class is $2$-divisible in $\operatorname{Pic}(S)$,
then $\sigma(S)$ is a positive multiple of $16$ by Rokhlin's theorem,
and examples with $\sigma(S)=16$ are constructed in \cite{LLR17}. It is
not known whether there exists a Kodaira fibred surface with $\sigma(S)
\leq 12$.

Kodaira fibred surfaces are a source of fascinating ad deep questions at the
cross-road between the algebro-geometric properties of a compact, complex
surface
and the topological properties of the underlying closed, oriented
$4$-manifold. In fact, they can be studied by using, besides the usual
algebro-geometric methods, techniques borrowed from geometric topology such
as the Meyer signature formula,
the Birman-Hilden relations in the mapping class group and the subtraction
of Lefchetz fibrations, see \cite{En98, EKKOS02, St02, L17}. We refer
the reader to the survey paper \cite{Cat17} and the references contained
therein for further details.

The original examples by Kodaira (see for instance \cite[Chapter V, Section 14]{BHPV03}) and its variants described in \cite{At69,
Hir69} are obtained by taking cyclic covers of a product of curves $C \times D$, branched over a smooth divisor which is the disjoint union of a finite number of graphs of regular maps $C \to D$.
Thus, they come with two distinct Kodaira fibrations, namely the pull-backs of the two natural fibrations in $C \times D$ (followed by a Stein factorization, if necessary).  This leads to the following definition
of ``double" Kodaira fibration, see \cite{Zaal95, LeBrun00, BDS01, BD02,
CatRol09, Rol10, LLR17}:
\begin{definition}
  A \emph{double Kodaira surface} is a compact, complex surface $S$, endowed
  with a \emph{double Kodaira fibration}, namely a surjective, holomorphic map
  $f \colon S \to B_1 \times B_2$ yielding, by composition with the natural
  projections, two Kodaira fibrations $f_i \colon S \to B_i$, $i=1, \, 2$.
\end{definition}

In the sequel, we will describe our approach to the construction of double
Kodaira
fibrations based on the techniques introduced in \cite{CaPol19, Pol20},
and present our results. The main step is to ``detopologize" the problem,
by transforming it into a purely algebraic one. This will be done in the
particular case of \emph{diagonal} double Kodaira fibrations, namely,
Stein factorizations of finite Galois covers
\begin{equation} \label{eq:intro-diagonal}
  \mathbf{f} \colon S \to \Sigma_b \times \Sigma_b,
\end{equation}
branched with order $n \geq 2$ over the diagonal $\Delta \subset \Sigma_b
\times \Sigma_b$, where $\Sigma_b$ is a closed Riemann surface of genus
$b$. By Grauert-Remmert's extension theorem and Serre's GAGA,
the existence of a $G$-cover $\mathbf{f}$ as in \eqref{eq:intro-diagonal}, up
to cover isomorphisms, is equivalent to the existence of a group epimorphism
\begin{equation} \label{eq:intr-varphi}
  \varphi \colon \pi_1(\Sigma_b \times \Sigma_b - \Delta) \to G,
\end{equation}
up to automorphisms of $G$. Furthermore, the condition that $\mathbf{f}$
is branched of order $n$ over $\Delta$
is rephrased by  asking that $\varphi(\gamma_{\Delta})$ has order $n$
in $G$, where $\gamma_{\Delta}$ is the homotopy class in $\Sigma_b \times
\Sigma_b - \Delta$ of a loop in $\Sigma_b \times \Sigma_b$ that ``winds once"
around $\Delta$. The requirement $n \geq 2$ means that $\varphi$ does not
factor through $\pi_1(\Sigma_b \times \Sigma_b)$; it also implies that $G$
is non-abelian, because $\gamma_{\Delta}$ is a non-trivial commutator in
$\pi_1(\Sigma_b \times \Sigma_b - \Delta)$. An epimorphism (or quotient) of type \eqref{eq:intr-varphi} such that $\varphi(\gamma_{\Delta})$ is non-trivial will be called \emph{admissible}.

Recall now that the group $\pi_1(\Sigma_b \times \Sigma_b-\Delta)$ is
isomorphic to
$\mathsf{P}_2(\Sigma_b)$, the pure braid group of genus $b$ on two strands, which 
admits a finite geometric presentation with $4b+1$ generators, see  \cite[Theorem
7]{GG04}. Taking the images of these generators via an admissible group epimorphism, we get an ordered set
\begin{equation} \label{eq:intro-ddks}
  \mathfrak{S}=(\r_{11}, \, \t_{11}, \ldots, \r_{1b}, \, \t_{1b}, \, \r_{21},
  \, \t_{21}, \ldots, \r_{2b}, \, \t_{2b}, \, \z)
\end{equation}
of $4b+1$ generators of $G$, such that $o(\z)=n$ and subject to a suitable finite set of relations. This will be called
a \emph{diagonal double Kodaira structure} of type $(b, \, n)$ on $G$,
see Definition \ref{def:dks}. Thus, the 
geometric problem of constructing  an admissible $G$-cover is translated into the combinatorial-algebraic
problem of finding a diagonal double Kodaira structure of type $(b, \, n)$
in $G$.

It turns out that the $G$-cover $\mathbf{f}$ is a diagonal double Kodaira
fibration (namely, the two surjective maps $f_i \colon S \to \Sigma_b$,
obtained as composition with the natural projections, have connected fibres)
if and only if the related structure $\mathfrak{S}$ is \emph{strong}, an additional condition
introduced in  Definition \ref{def:strongdks}; furthermore, the algebraic
signature $\sigma(\mathfrak{S})$, introduced in Definition \ref{def:signature-ddks},
equals the geometric signature $\sigma(S)$. 

Note that not every double Kodaira fibration is of diagonal type. In fact,
  one proves that  if $S$
  is of diagonal type then its slope satisfies $\nu(S)=2+s$, where
  $s$ is rational and	$0 <s<6-4\sqrt{2}$, and that there exist examples whose high slope violates  this inequality (for instance, Catanese-Rollenske's example with $\nu(S)=2+2/3$); see  \cite[Section 4]{Pol20}. For more details on the construction of diagonal double Kodaira fibrations, we refer the reader to Section \ref{sec:ddkf}.

In the light of the previous considerations, classifying diagonal double Kodaira fibrations is equivalent to
describing finite groups which admit a diagonal double Kodaira structure. Our key result in this direction is the following: 

\begin{A}[see Propositions \ref{prop:less-32-no-ddks},
		\ref{prop:32-no-extra-special-no-structure} and Theorem
\ref{thm:main-algebraic}]
  Let $G$ be a finite group admitting a diagonal double Kodaira
  structure. Then  $|G| \geq 32$, with equality if and only if $G$
  is extra-special $($see $\operatorname{Section}$
  $\operatorname{\ref{sec:CCT} }$ for the definition$)$. Moreover, the
  following holds.
  \begin{itemize}
    \item[$\boldsymbol{(1)}$] Both extra-special groups $G$ of order $32$
      admit $2211840=1152 \cdot 1920$ diagonal double Kodaira structures of
      type $(b, \, n)=(2, \, 2)$. Every such a structure
      $\mathfrak{S}$ is strong and satisfies $\sigma(\mathfrak{S})=16$.
    \item[$\boldsymbol{(2)}$] If $G=G(32, \, 49)=\mathsf{H}_5(\mathbb{Z}_2)$,
      these structures form $1920$ orbits under the action of
      $\operatorname{Aut}(G)$.
    \item[$\boldsymbol{(3)}$] If $G=G(32, \, 50)=\mathsf{G}_5(\mathbb{Z}_2)$,
      these structures form $1152$ orbits under the action of
      $\operatorname{Aut}(G)$.
  \end{itemize}
\end{A}

Our Main Theorem should be compared with previous results, obtained by the
first author in collaboration with A. Causin, regarding the construction of
diagonal double Kodaira structures on some extra-special groups of order at
least $2^7=128$, see \cite{CaPol19, Pol20}. However, even if the definition of diagonal double Kodaira structure and the construction of the corresponding diagonal double Kodaira fibration  presented in Sections \ref{sec:ddks} and \ref{sec:ddkf} closely follow the ones in \cite{Pol20}, the examples constructed here are really new, in the sense that they cannot be obtained
as images of structures on extra-special groups of larger order (Remark \ref{rmk:comparison-minimal-order}). It is precisely the original part of this paper, namely the subtle group theoretical analysis developed in Sections \ref{sec:CCT} and \ref{sec:DDKS} and used in the proof of the Main Theorem, which allows us to pass from $|G|=128$ to $|G|=32$.

\smallskip

The interpretation of the Main Theorem in terms of admissible epimorphisms from surface braid groups
to finite groups  is  given in Corollary \ref{cor:quotient-of-P2}. As a consequence, we can describe all diagonal double Kodaira fibrations associated with structures of type $(2, \, 2)$ on extra-special groups of order $32$ (Theorem \ref{thm:main-geometric}), showing that they provide the sharp lower bound  
$\sigma(S) \geq 16$ for the signature of a diagonal double Kodaira fibration (Corollary \ref{cor:bound-signature}).

\smallskip

These results yield, as a by-product, new ``double solutions'' to a problem
(stated by G. Mess) from Kirby's problem list in low-dimensional topology
\cite[Problem 2.18 A]{Kir97}, asking what is the smallest number $b$ for
which there exists a real surface bundle over a real surface with base genus
$b$ and non-zero signature. We actually have $b=2$, also for double Kodaira
fibrations, as shown in \cite[Proposition 3.19]{CaPol19} and \cite[Theorem
4.6]{Pol20} by using double Kodaira structures of type $(2, \, 3)$ on
extra-special groups of order $3^5$. Those fibrations had signature $144$
and fibre genera $325$; the new examples presented here substantially lower both these
values, in fact they have signature $16$ and fibre genera $41$ (Theorem \ref{thm:new-Kirby}).


\smallskip

We believe that the results described above are significant for at least
two reasons.
\begin{itemize}
  \item[$\boldsymbol{(i)}$] Although we know that $\mathsf{P}_2(\Sigma_b)$
    is residually $p$-finite for all prime number $p \geq 2$, see
    \cite[pp. 1481-1490]{BarBel09}, so far there has been no systematic
    work aimed to describe its admissible finite quotients. The first results in
    this direction were those of A. Causin and the first author, who showed
    that both extra-special groups of order $p^{4b+1}$
    appear as admissible quotients of $\mathsf{P}_2(\Sigma_b)$ for all $b
    \geq 2$ and
    all prime numbers $p \geq 5$; moreover, if $p$ divides $b+1$, then both
    extra-special groups of order $p^{2b+1}$ appear as admissible quotients,
    too. 
    Our work sheds some new light on this problem, by providing a
    sharp lower
    bound for the order of an admissible quotient. Moreover, for both
    extra-special groups
    of order $32$ (namely, the ones for which the bound is attained) we are able to compute the number of admissible epimorphisms
    $\varphi \colon \mathsf{P}_2(\Sigma_2) \to G$,
    and the number of their equivalence classes up to the natural action of
    $\operatorname{Aut}(G)$.
  \item[$\boldsymbol{(ii)}$] Constructing (double) Kodaira fibrations with
    small signature is
    a rather difficult problem, and there are few examples in the literature (\cite{BD02, LLR17}). As far as we know, the present paper provides the first positive-dimensional families of such examples, see Remark \ref{rmk:comparison-with-loenne} for more details.
\end{itemize}

\medskip
Let us now describe how this paper is organized. In Section \ref{sec:CCT}
we introduce some algebraic preliminaries, in particular we discuss
the so-called CCT-groups (Definition \ref{def:CCT}), namely, finite
non-abelian groups in which commutativity is a transitive relation on the
set of non-central elements. These groups are of historical importance
in the context of classification of finite simple groups, see Remark
\ref{rmk:CCT-historical-importance}, and they play a relevant role in
this paper. It turns out that there are precisely
eight groups $G$
with $|G| \leq 32$ that are not CCT-groups, namely $\mathsf{S}_4$ and
seven groups of order $32$, see Corollary \ref{cor:small-CCT}, Proposition
\ref{prop:24-no-CCT} and Proposition \ref{prop:CCT-minore-32}.

In Section \ref{sec:ddks} we define diagonal double Kodaira structures
 and we explain the relation with their counterpart in
geometric topology, namely admissible group epimorphisms from  pure surface
braid groups to finite groups.

Section \ref{sec:DDKS} is devoted to the study of diagonal double Kodaira
structures in groups of order at most $32$. One crucial technical result
is Proposition \ref{prop:no-structure-CCT}, stating that there are no such
structures on CCT-groups. Thus, in order to prove the first part of the Main Theorem, we only need to exclude the existence of
diagonal double Kodaira structures on $\mathsf{S}_4$ and on the five
non-abelian, non-CCT groups
of order $32$; this is done in	Proposition \ref{prop:less-32-no-ddks}
and Proposition \ref{prop:32-no-extra-special-no-structure}, respectively.
The second part of the Main Theorem, i.e. the computation of number of structures in
each case, is obtained by using some techniques borrowed from \cite{Win72};
more precisely, we exploit the fact that  $V=G/Z(G)$ is a symplectic vector
space of dimension $4$ over $\mathbb{Z}_2$,  and that $\operatorname{Out}(G)$
embeds in $\mathrm{Sp}(4, \, \mathbb{Z}_2)$ as the orthogonal group
associated with the quadratic form $q \colon V \to \mathbb{Z}_2$ related
to the symplectic form $(\cdot \;, \cdot)$ by $q(\overline{\mathsf{x}}\,
\overline{\mathsf{y}})=q(\overline{\mathsf{x}}) + q(\overline{\mathsf{y}})+
(\overline{\mathsf{x}}, \, \overline{\mathsf{y}})$.

Finally, in Section \ref{sec:ddkf} we establish the relation between our
algebraic
results and  the existence of diagonal double Kodaira fibrations,
and we prove the consequences of the Main Theorem in this geometrical framework. 

The paper ends with an Appendix, where we collect the
presentations for the non-abelian groups of order $24$ and $32$ that we
used in our calculations.

\bigskip

$\mathbf{Notation \; and \; conventions}$. If $S$ is a complex, non-singular
projective surface, then $c_1(S)$, $c_2(S)$ denote the first and second
Chern class of its tangent bundle $T_S$, respectively. 


Throughout the paper we use the following notation for
groups:
\begin{itemize}
  \item $\mZ_n$: cyclic group of order $n$.
  \item $G=N \rtimes Q$: semi-direct product of $N$ and $Q$, namely, split
    extension of $Q$ by $N$, where $N$ is normal in $G$.
  \item $G=N.Q$: non-split extension of $Q$ by $N$.
  \item $\operatorname{Aut}(G)$: the automorphism group of $G$.
  \item $\mathsf{D}_{p, \, q, \, r}=\mathbb{Z}_q \rtimes \mathbb{Z}_p=
    \langle
    x, \, y \; |
    \; x^p=y^q=1, \; xyx^{-1}=y^r \rangle$: split metacyclic
    group of order
    $pq$. The group
    $\mathsf{D}_{2, \, n, \, -1}$ is the dihedral group of
    order $2n$ and will
    be denoted by $\mathsf{D}_{2n}$.
  \item If $n$ is an integer greater or equal to $4$, we denote by
    $\mathsf{QD}_{2^n}$ the quasi-dihedral group of order
    $2^n$, having
    presentation
    \begin{equation}
      \mathsf{QD}_{2^n} := \langle x, \, y \mid
      x^2=y^{2^{n-1}} = 1, \; xyx^{-1}
      = y^{2^{n-2} - 1} \rangle.
    \end{equation}
  \item The generalized quaternion group of order $4n$ is denoted by
    $\mathsf{Q}_{4n}$ and is presented as
    \begin{equation}
      \mathsf{Q}_{4n} = \langle x, \,y, \, z \mid x^n =
      y^2 = z^2 = xyz \rangle.
    \end{equation}
    For $n=2$ we obtain the usual quaternion group $\mathsf{Q}_8$,
    for which we
    adopt the classical presentation
    \begin{equation}
      \mathsf{Q}_{8}=\langle i,\,j,\,k \mid i^2 = j^2 =
      k^2 = ijk\rangle,
    \end{equation}
    denoting by $-1$ the unique element of order $2$.
  \item $\mathsf{S}_n, \;\mathsf{A}_n$: symmetric, alternating group
    on $n$
    symbols. We write the
    composition of permutations from the right to the left;
    for instance,
    $(13)(12)=(123)$.
  \item $\mathsf{GL}(n, \, \mathbb{F}_q), \, \mathsf{SL}(n, \,
    \mathbb{F}_q), \, \mathsf{Sp}(n, \, \mathbb{F}_q)$: general linear
    group, special linear group and symplectic group of $n
    \times n$ matrices over a field with $q$ elements.
  \item The order of a finite group $G$ is denoted by $|G|$. If $x
    \in G$,
    the order of
    $x$ is denoted by $o(x)$ and its centralizer in $G$
    by $C_G(x)$.
  \item If $x, \, y \in G$, their commutator is defined as
    $[x,\, y]=xyx^{-1}y^{-1}$.
  \item The commutator subgroup of $G$ is denoted by $[G, \, G]$,
    the center
    of $G$ by $Z(G)$.
  \item If $S= \{s_1, \ldots, s_n \} \subset G$, the subgroup generated
    by $S$ is denoted by $\langle S \rangle=\langle s_1,\ldots,
    s_n \rangle$.
  \item $\mathrm{IdSmallGroup}(G)$ indicates the label of the group
    $G$ in
    the  \verb|GAP4| database of small groups. For instance
    $\mathrm{IdSmallGroup}(\mathsf{D}_4)=G(8, \, 3)$ means
    that $\mathsf{D}_4$
    is the third in the
    list of groups of order $8$.
  \item If $N$ is a normal subgroup of $G$ and $g \in G$, we denote by
  $\bar{g}$
    the image of $g$ in the quotient group $G/N$.
\end{itemize}

\section{Group-theoretical preliminaries: CCT-groups and extra-special groups}
\label{sec:CCT}

\begin{definition} \label{def:CCT}
  A finite non-abelian group $G$ is said to be a
  \emph{center commutative-transitive group}
  $($or a CCT-\emph{group}, for short$)$ if commutativity is a transitive
  relation on the set on non-central elements of $G$. In other words, if $x,
  \, y, \, z	\in G \setminus Z(G)$  and $[x, \, y]=[y, \, z]=1$, 
  then $[x, \, z]=1$.
\end{definition}
Other characterizations of CCT-groups are provided in the statement below, whose proof is straightforward.

\begin{proposition} \label{prop:CCT}
  For a finite group $G$, the following properties are equivalent.
  \begin{itemize}
    \item[$\boldsymbol{(1)}$] $G$ is a \emph{CCT}-group.
    \item[$\boldsymbol{(2)}$] For every pair $x, \, y $ of
      non-central elements in $G$, the relation $[x, \, y]=1$ implies
      $C_G(x)=C_G(y)$.
    \item[$\boldsymbol{(3)}$] For every non-central element $x
      \in G$, the centralizer $C_G(x)$ is abelian.
  \end{itemize}
\end{proposition}

\begin{remark} \label{rmk:CCT-historical-importance}
  CCT-groups are of historical importance in the context of classification of
  finite simple groups, see for instance \cite{Suz61}, where they are called
  CA-\emph{groups}. Further references are  \cite{Schm70},
  \cite{Reb71}, \cite{Rocke73}, \cite{Wu98}.
\end{remark}

\begin{lemma} \label{lem:G/Z(G) cyclic}
  If $G$ is a finite group such that $G/Z(G)$ is cyclic, then $G$
  is abelian.
\end{lemma}
\begin{proof}
This is a standard exercise, cf. \cite[Problem 12 p. 77]{Her64}. 
\end{proof}

\begin{proposition} \label{prop:small-CCT}
  Let $G$ be a finite non-abelian group.
  \begin{itemize}
    \item[$\boldsymbol{(1)}$] If $|G|$ is the product of at
      most three prime
      factors $($non necessarily distinct$)$, then $G$
      is a \emph{CCT}-group.
    \item[$\boldsymbol{(2)}$] If $|G|=p^4$, with $p$ prime,
      then $G$ is a
      \emph{CCT}-group.
    \item[$\boldsymbol{(3)}$] If $G$ contains an abelian normal
      subgroup of
      prime index, then $G$ is a \emph{CCT}-group.
  \end{itemize}
\end{proposition}

\begin{proof}
  $\boldsymbol{(1)}$ Assume that $|G|$ is the product of at most
  three prime
  factors, and take a non-central element $y$. Then the centralizer
  $C_G(y)$
  has non-trivial center, because $1 \neq y \in C_G(y)$, and its
  order is the
  product of at most two primes. Therefore the quotient of  $C_G(y)$
  by its
  center is cyclic, hence $C_G(y)$ is abelian by Lemma \ref{lem:G/Z(G)
  cyclic}.

  $\boldsymbol{(2)}$ Assume $|G|=p^4$ and suppose by contradiction
  that there
  exist three elements $x, \, y, \, z \in G\setminus Z(G)$ 
  such that $[x, \, y]=[y, \, z]=1$
  but $[x,
  \, z] \neq 1$. They generate a non-abelian subgroup $N=\langle x, \, y, \,
  z \rangle$, which is not the whole of $G$ since $y \in Z(N)$ but $y \notin
  Z(G)$. It
  follows that $N$ has order $p^3$ and so, by Lemma
  \ref{lem:G/Z(G) cyclic}, its center is cyclic of order $p$,
  generated by
  $y$. The group $G$ is a finite $p$-group, hence a nilpotent group;
  being a
  proper subgroup of maximal order in a nilpotent group, $N$ is normal
  in $G$
  (see \cite[Corollary 5.2]{Mac12}), so we have a conjugacy
  homomorphism $G
  \to \mathrm{Aut}(N)$, that in turn induces a conjugacy homomorphism
  $G \to
  \mathrm{Aut}(Z(N)) \simeq \mathbb{Z}_{p-1}$. The image of such
  a homomorphism
  must have order dividing both $p^4$ and $p-1$, hence it is
  trivial. In other
  words, the conjugacy action of $G$ on $Z(N)=\langle y \rangle$
  is trivial,
  hence $y$ is central in $G$, contradiction.

  $\boldsymbol{(3)}$ Let $N$ be an abelian normal subgroup of $G$ such
  that $G/N$
  has prime order $p$. As $G/N$ has no non-trivial proper subgroups, it
  follows that
  $N$ is
  a maximal subgroup of $G$. Let $x$ be any non-central element
  of $G$,
  so that $C_G(x)$ is a proper subgroup of $G$; then there are two
  possibilities. \\ \\
  \underline{Case 1}: $x \in N$. Then $N \subseteq C_G(x)$
  and so, by the
  maximality of $N$, we get  $C_G(x)=N$, which is
  abelian. \\
  \underline{Case 2}: $x \notin N$. Then the image of $x$
  generates $G/N$, and
  so every element $y \in G$ can be written in the
  form $y=ux^r$, where $u \in
  N$ and $0 \leq r \leq p-1$. In particular, if $y
  \in C_G(x)$, the condition
  $[x,\, y]=1$ yields $[x, \, u]=1$, namely $u \in
  N \cap C_G(x)$. Since $N$
  is abelian, it follows that $C_G(x)$
  is abelian, too.

\end{proof}

We now want to classify non-abelian,  non-CCT groups of order at most
$32$. First of all,
as an immediate consequence of parts $\boldsymbol{(1)}$ and $\boldsymbol{(2)}$
of Proposition \ref{prop:small-CCT}, we have the following
\begin{corollary}  \label{cor:small-CCT}
  Let  $G$ be a finite non-abelian group such that $|G| \leq
  32$. If $G$
  is not a \emph{CCT}-group, then either $|G|=24$ or $|G|=32$.
\end{corollary}

Let us start by considering the case $G=24$.

\begin{proposition} \label{prop:24-no-CCT}
  Let $G$ be a finite non-abelian group such that $|G|=24$ and $G$
  is not a \emph{CCT}-group. Then $G=\mathsf{S}_4$.
\end{proposition}
\begin{proof}
  We start by observing that $\mathsf{S}_4$ is not a CCT-group. In fact,
  $(1234)$  commutes to its square $(13)(24)$, which commutes to
  $(12)(34)$, but $(1234)$ and $(12)(34)$ do not commute.

  We are left to show that the remaining non-abelian groups of order
  $24$ are all CCT-groups; we will do a case-by-case analysis, referring
  the reader to the presentations given in Table \ref{table:24-nonabelian}
  of Appendix A. Apart
  from $G=G(24, \, 3)= \mathsf{SL}(2, \, \mathbb{F}_3)$, for which we give
  an ad-hoc proof, we will show that all these groups contain an abelian
  subgroup $N$ of prime index, so that we can conclude by using part
  $\boldsymbol{(3)}$ of Proposition \ref{prop:small-CCT}.
  \begin{itemize}
    \item $G=G(24, \, 1).$ Take $N=\langle x^2y \rangle \simeq
      \mathbb{Z}_{12}$.
    \item $G=G(24, \, 3).$ The action of $\mathrm{Aut}(G)$
      has five orbits,
      whose representative elements are $\{1, \, x, \,
      x^2, \, z, \, z^2\}$, see
      \cite{SL(23)}. We have $\langle z^2 \rangle = Z(G)$
      and so, since $C_G(x)
      \subseteq C_G(x^2)$, it suffices to show that the
      centralizers of $x^2$
      and $z$ are both abelian. In fact, we have
      \begin{equation}
	C_G(x^2)=\langle x \rangle \simeq
	\mathbb{Z}_6, \quad C_G(z)=\langle z
	\rangle \simeq \mathbb{Z}_4.
      \end{equation}
    \item $G=G(24, \, 4).$ Take $N=\langle x \rangle \simeq
      \mathbb{Z}_{12}$.
    \item $G=G(24, \, 5).$ Take $N=\langle y \rangle \simeq
      \mathbb{Z}_{12}$.
    \item $G=G(24, \, 6).$ Take $N=\langle y \rangle \simeq
      \mathbb{Z}_{12}$.
    \item $G=G(24, \, 7).$ Take $N=\langle z, \, x^2y \rangle
      \simeq \mathbb{Z}_6
      \times \mathbb{Z}_2$.
    \item $G=G(24, \, 8).$ Take $N=\langle y, \, z, \, w
      \rangle \simeq
      \mathbb{Z}_6 \times \mathbb{Z}_2$.
    \item $G=G(24, \, 10).$ Take $N=\langle z, \, y \rangle \simeq
      \mathbb{Z}_{12}$.
    \item $G=G(24, \, 11).$ Take $N=\langle z, \, i \rangle \simeq
      \mathbb{Z}_{12}$.
    \item $G=G(24, \, 13).$ Take $N=\langle z \rangle \times
      \mathsf{V}_4\simeq
      (\mathbb{Z}_2)^3$, where $\mathsf{V}_4 = \langle (1\,2)(3\,4), \;
      (1\, 3)(2\, 4) \rangle$ is the Klein
      subgroup.
    \item $G=G(24, \, 14).$ Take $N=\langle z, \, w \rangle
      \times \langle (123)
      \rangle \simeq \mathbb{Z}_6 \times \mathbb{Z}_2$.
  \end{itemize}
  This completes the proof.
\end{proof}

The next step is to classify non-abelian, non-CCT groups $G$
with $|G| = 32$; it will turn out that there are
precisely seven of them, see Proposition \ref{prop:CCT-minore-32}. Before
doing
this, let us introduce the following classical definition, see for instance
\cite[p. 183]{Gor07} and \cite[p. 123]{Is08}.

\begin{definition} \label{def:extra-special}
  Let $p$ be a prime number. A finite $p$-group $G$ is called
  \emph{extra-special} if its center $Z(G)$ is cyclic of order $p$
  and the
  quotient $V=G/Z(G)$ is a non-trivial, elementary abelian $p$-group.
\end{definition}

An elementary abelian $p$-group is a finite-dimensional vector space over
the field $\mathbb{Z}_p$, hence it is of the form $V=(\mathbb{Z}_p)^{\dim
V}$ and $G$ fits into a short exact sequence
\begin{equation} \label{eq:extension-extra}
  1 \to \mathbb{Z}_p \to G \to V \to 1.
\end{equation}
Note that, $V$ being abelian, we must have $[G, \, G]=\mathbb{Z}_p$, namely
the commutator subgroup of $G$ coincides with its center. Furthermore,
since the extension \eqref{eq:extension-extra} is central, it cannot be
split, otherwise $G$ would be isomorphic to the direct product of the two
abelian groups $\mathbb{Z}_p$ and $V$, which is impossible because $G$
is non-abelian.

If $G$ is extra-special, then we can define a map $\omega \colon V \times
V \to \mathbb{Z}_p$ as follows: for every $v_1, \, v_2 \in V$, we set
$\omega(v_1, \, v_2)=[g_1, \, g_2]$, where $g_i$ is any lift of $v_i$
in $G$. This turns out to be a symplectic form on $V$, hence
$\dim V$ is even and $|G|=p^{\dim V +1}$ is an odd power of $p$.

For every prime number $p$, there are precisely two isomorphism classes
$M(p)$, $N(p)$ of non-abelian groups of order $p^3$, namely
\begin{equation*}
  \begin{split}
    M(p)& = \langle \r, \, \t, \, \z \; | \; \r^p=\t^p=1, \,
    \z^p=1, [\r, \,
    \z]=[\t,\, \z]=1, \, [\r, \, \t]=\z^{-1} \rangle \\
    N(p)& = \langle \r, \, \t, \, \z \; | \; \r^p=\t^p=\z, \,
    \z^p=1, [\r, \,
    \z]=[\t,\, \z]=1, \, [\r, \, \t]=\z^{-1} \rangle \\
  \end{split}
\end{equation*}
and both of them are in fact extra-special, see \cite[Theorem 5.1 of
Chapter 5]{Gor07}.

If $p$ is odd, then the groups $M(p)$ and $N(p)$ are distinguished by
their exponent, which equals $p$ and $p^2$, respectively. If $p=2$, the
group $M(p)$ is isomorphic to the dihedral group $D_8$, whereas $N(p)$
is isomorphic to the quaternion group $\mathsf{Q}_8$.

The classification of extra-special $p$-groups is now provided by the result
below, see \cite[Section 5 of Chapter 5]{Gor07} and \cite[Section 2]{CaPol19}.
\begin{proposition} \label{prop:extra-special-groups}
  If $b \geq 2$ is a positive integer and $p$ is a prime number, there
  are exactly two isomorphism classes of extra-special $p$-groups
  of order
  $p^{2b+1}$, that can be described as follows.
  \begin{itemize}
    \item The central product $\mathsf{H}_{2b+1}(\mathbb{Z}_p)$
      of $b$ copies
      of $M(p)$, having presentation
      \begin{equation} \label{eq:H5}
	\begin{split}
	  \mathsf{H}_{2b+1}(\mathbb{Z}_p)
	  = \langle \, & \mathsf{r}_1, \,
	  \mathsf{t}_1,
	  \ldots, \mathsf{r}_b,\, \mathsf{t}_b,
	  \, \z \; |  \; \mathsf{r}_{j}^p =
	  \mathsf{t}_{j}^p=\mathsf{z}^p=1, \\
	  & [\mathsf{r}_{j}, \, \mathsf{z}]  =
	  [\mathsf{t}_{j}, \, \mathsf{z}]= 1, \\
	  & [\mathsf{r}_j, \, \mathsf{r}_k]=
	  [\mathsf{t}_j, \, \mathsf{t}_k] =
	  1, \\
	  & [\mathsf{r}_{j}, \,\mathsf{t}_{k}]
	  =\mathsf{z}^{- \delta_{jk}} \, \rangle.
	\end{split}
      \end{equation}
      If $p$ is odd, this group has exponent $p$ and is isomorphic to the
      matrix Heisenberg group $\mathcal{H}_{2b+1}(\mathbb{Z}_p) \subset
      \mathsf{GL}(b+2, \, \mathbb{Z}_p)$ of dimension $2b+1$ over the
      field $\mathbb{Z}_p$.
    \item The central product $\mathsf{G}_{2b+1}(\mathbb{Z}_p)$
      of $b-1$ copies
      of $M(p)$ and one copy of $N(p)$,  having presentation
      \begin{equation} \label{eq:G5}
	\begin{split}
	  \mathsf{G}_{2b+1}(\mathbb{Z}_p)
	  = \langle \, & \mathsf{r}_1, \,
	  \mathsf{t}_1,
	  \ldots, \mathsf{r}_b,\, \mathsf{t}_b,
	  \, \z \; | \;  \mathsf{r}_{b}^p =
	  \mathsf{t}_{b}^p=\mathsf{z}, \\
	  &\mathsf{r}_{1}^p = \mathsf{t}_{1}^p=
	  \ldots = \mathsf{r}_{b-1}^p =
	  \mathsf{t}_{b-1}^p=\mathsf{z}^p=1, \\
	  & [\mathsf{r}_{j}, \, \mathsf{z}]  =
	  [\mathsf{t}_{j}, \, \mathsf{z}]= 1, \\
	  & [\mathsf{r}_j, \, \mathsf{r}_k]=
	  [\mathsf{t}_j, \, \mathsf{t}_k] =
	  1, \\
	  & [\mathsf{r}_{j}, \,\mathsf{t}_{k}]
	  =\mathsf{z}^{- \delta_{jk}} \, \rangle.
	\end{split}
      \end{equation}
      If $p$ is odd, this group has exponent $p^2$.
  \end{itemize}
\end{proposition}

\begin{remark} \label{rmk:further-commutators}
  In both cases, from the relations above we deduce
  \begin{equation} \label{eq:further-commutators}
    [\mathsf{r}_j^{-1}, \, \mathsf{t}_k]=\mathsf{z}^{\delta_{jk}},
    \quad
    [\mathsf{r}_j^{-1}, \,
    \mathsf{t}_k^{-1}]=\mathsf{z}^{-\delta_{jk}}
  \end{equation}
\end{remark}

\begin{remark} \label{rmk:center-H-G}
  For both groups $\mathsf{H}_{2b+1}(\mathbb{Z}_p)$ and
  $\mathsf{G}_{2b+1}(\mathbb{Z}_p)$, the center coincides with
  the derived
  subgroup and is equal to $\langle  \z  \rangle \simeq \mathbb{Z}_p$. Note that,
  being these groups non-abelian,  this condition implies that their nilpotency class is
  $2$, see \cite[p. 22]{Is08}.
\end{remark}

\begin{remark}
  If $p=2$, we can distinguish the two groups
  $\mathsf{H}_{2b+1}(\mathbb{Z}_p)$
  and $\mathsf{G}_{2b+1}(\mathbb{Z}_p)$ by counting the number
  of elements
  of order $4$.
\end{remark}

\begin{remark} \label{rmk:extra-special-not-CCT}
  The groups $\mathsf{H}_{2b+1}(\mathbb{Z}_p)$
  and $\mathsf{G}_{2b+1}(\mathbb{Z}_p)$ are not \textrm{CCT}-groups. In
  fact, let us take two distinct indices $j, \, k \in \{1, \ldots, b \}$
  and consider the non-central elements $\mathsf{r}_j$, $\mathsf{t}_j$,
  $\mathsf{t}_k$. Then we have
  $[\mathsf{r}_j, \, \mathsf{t}_k]=[\mathsf{t}_k, \,
  \mathsf{t}_j]=1$, but $[\mathsf{r}_j, \, \mathsf{t}_j]=\mathsf{z}^{-1}$.
\end{remark}

We can now analyze the case $|G|=32$.
\begin{proposition} \label{prop:CCT-minore-32}
  Let $G$ be a finite non-abelian group such that $|G|=32$ and $G$
  is not a
  \emph{CCT}-group. Then $G=G(32, \, t)$, where $t \in \{6, \, 7,
  \, 8, \, 43, \, 44, \, 49, \, 50 \}$. Here	$G(32, \,
  49)=\mathsf{H}_5(\mathbb{Z}_2)$ and  $G(32,
  \,50)=\mathsf{G}_5(\mathbb{Z}_2)$
  are the two extra-special groups of order $32$.
\end{proposition}
\begin{proof}
  We first do a case-by case analysis showing that, if $t \notin \{6, \,
  7, \, 8, \, 43, \, 44, \, 49, \, 50 \}$, then $G=G(32, \, t)$ contains
  an abelian subgroup $N$ of index $2$, so that $G$ is a CCT-group
  by part $\boldsymbol{(3)}$ of Proposition \ref{prop:small-CCT}. In every
  case, we refer the reader to the presentation given in Table
  \ref{table:32-nonabelian} of Appendix A.
  \begin{itemize}
    \item $G=G(32, \, 2).$ Take $N=\langle x, \, y^2, z \rangle
      \simeq \mathbb{Z}_4
      \times (\mZ_2)^2$.
    \item $G=G(32, \, 4).$ Take $N=\langle x, \, y^2 \rangle
      \simeq (\mZ_4)^2$.
    \item $G=G(32, \, 5).$ Take $N= \langle x, \, y \rangle
      \simeq \mZ_8 \times
      \mZ_2$.
    \item $G=G(32, \, 9).$ Take $N= \langle x, \, y \rangle
      \simeq \mZ_8 \times
      \mZ_2$.
    \item $G=G(32, \, 10).$ Take $N= \langle ix, \, k \rangle
      \simeq \mZ_8
      \times \mZ_2$.
    \item $G=G(32, \, 11).$ Take $N= \langle x, \, y \rangle
      \simeq (\mZ_4)^2$.
    \item $G=G(32, \, 12).$ Take $N= \langle x^2, \, y \rangle
      \simeq (\mZ_4)^2$.
    \item $G=G(32, \, 13).$ Take $N= \langle x^2, \, y \rangle
      \simeq \mZ_8
      \times \mZ_2$.
    \item $G=G(32, \, 14).$ Take $N= \langle x^2, \, y \rangle
      \simeq \mZ_8
      \times \mZ_2$.
    \item $G=G(32, \, 15).$ Take $N= \langle x^2, \, y \rangle
      \simeq \mZ_8
      \times \mZ_2$.
    \item $G=G(32, \, 17).$ Take $N=\langle y \rangle \simeq
      \mZ_{16}$.
    \item $G=G(32, \, 18).$ Take $N=\langle y \rangle \simeq
      \mZ_{16}$.
    \item $G=G(32, \, 19).$ Take $N=\langle y \rangle \simeq
      \mZ_{16}$.
    \item $G=G(32, \, 20).$ Take $N=\langle x \rangle \simeq
      \mZ_{16}$.
    \item $G=G(32, \, 22).$ Take $N=\langle w  \rangle \times
      \langle x, \,
      y \rangle \simeq \mZ_{8} \times (\mZ_2)^2$.
    \item $G=G(32, \, 23).$ Take $N=\langle z  \rangle \times
      \langle x, \,
      y^2 \rangle \simeq \mZ_4 \times (\mZ_2)^2$.
    \item $G=G(32, \, 24).$ Take $N= \langle x, \, y \rangle
      \simeq (\mZ_4)^2$.
    \item $G=G(32, \, 25).$ Take $N=\langle z  \rangle \times
      \langle y^2 \rangle
      \simeq (\mZ_4)^2$.
    \item $G=G(32, \, 26).$ Take $N=\langle z  \rangle \times
      \langle i \rangle
      \simeq (\mZ_4)^2$.
    \item $G=G(32, \, 27).$ Take $N=\langle x, \, y, \, a, \,
      b  \rangle \simeq
      (\mZ_2)^4$.
    \item $G=G(32, \, 28).$ Take $N=\langle x, \, y, \, z
      \rangle \simeq
      \mathbb{Z}_4 \times  (\mathbb{Z}_2)^2$.
    \item $G=G(32, \, 29).$ Take $N=\langle x, \, i, \, z
      \rangle \simeq
      \mathbb{Z}_4 \times  (\mathbb{Z}_2)^2$.
    \item $G=G(32, \, 30).$ Take $N=\langle x, \, y, \, z
      \rangle \simeq
      \mathbb{Z}_4 \times  (\mathbb{Z}_2)^2$.
    \item $G=G(32, \, 31).$ Take $N= \langle x, \, y \rangle
      \simeq (\mZ_4)^2$.
    \item $G=G(32, \, 32).$ Take $N= \langle y, \, z \rangle
      \simeq (\mZ_4)^2$.
    \item $G=G(32, \, 33).$ Take $N= \langle x, \, y \rangle
      \simeq (\mZ_4)^2$.
    \item $G=G(32, \, 34).$ Take $N= \langle x, \, y \rangle
      \simeq (\mZ_4)^2$.
    \item $G=G(32, \, 35).$ Take $N=\langle x, \, k \rangle
      \simeq (\mZ_4)^2$.
    \item $G=G(32, \, 37).$ Take $N= \langle x, \, y \rangle
      \simeq \mZ_8
      \times \mZ_2$.
    \item $G=G(32, \, 38).$ Take $N= \langle x, \, y \rangle
      \simeq \mZ_8
      \times \mZ_2$.
    \item $G=G(32, \, 39).$ Take $N=\langle z  \rangle \times
      \langle y \rangle
      \simeq \mZ_{8} \times \mZ_2$.
    \item $G=G(32, \, 40).$ Take $N=\langle z  \rangle \times
      \langle y \rangle
      \simeq \mZ_{8} \times \mZ_2$.
    \item $G=G(32, \, 41).$ Take $N=\langle w  \rangle \times
      \langle x \rangle
      \simeq \mZ_{8} \times \mZ_2$.
    \item $G=G(32, \, 42).$ Take $N= \langle x, \, y \rangle
      \simeq \mZ_8
      \times \mZ_2$.
    \item $G=G(32, \, 46).$ Take $N=\langle z, \,w	\rangle
      \times \langle y
      \rangle \simeq \mZ_4 \times (\mZ_2)^2$.
    \item $G=G(32, \, 47).$ Take $N=\langle z, \,w	\rangle
      \times \langle i
      \rangle \simeq \mZ_4 \times (\mZ_2)^2$.
    \item $G=G(32, \, 48).$ Take $N=\langle x, \, y, \, z
      \rangle \simeq
      \mathbb{Z}_4 \times  (\mathbb{Z}_2)^2$.
  \end{itemize}
  It remains to show that $G=G(32, \, t)$ is not a \textrm{CCT}-group
  for
  $t \in \{6, \, 7, \, 8, \, 43, \, 44, \, 49, \, 50 \}$.

  For $t=49$ and $t=50$ we have the two extra-special cases, that
  are not
  \textrm{CCT}-groups by Remark \ref{rmk:extra-special-not-CCT}. Let us
  now deal
  with the remaining values of $t$: for each of them, we will exhibit three
  non-central
  elements for which commutativity is not a transitive relation, and this will complete the proof.
  \begin{itemize}
    \item $G=G(32, \, 6).$ The center of $G$ is $Z(G)= \langle
      x \rangle \simeq
      \mathbb{Z}_2$. We have $[y, \, w^2]=[w^2, \, w]=1$,
      but $[y, \, w]=x$. 
    \item $G=G(32, \, 7).$ The center of $G$ is $Z(G)= \langle
      w \rangle \simeq
      \mathbb{Z}_2$. We have $[y, \, z]=[z, \, u]=1$,
      but $[y, \, u]=w$. 
    \item $G=G(32, \, 8).$ The center of $G$ is $Z(G)= \langle
      x^4 \rangle \simeq
      \mathbb{Z}_2$. We have $[x, \, x^2]=[x^2, \, y]=1$,
      but $[x, \, y]=z^2$. 
    \item $G=G(32, \, 43).$ The center of $G$ is $Z(G)= \langle
      x^4 \rangle \simeq
      \mathbb{Z}_2$. We have $[x, \, x^2]=[x^2, \, z]=1$,
      but $[x, \, z]=x^4$.
    \item $G=G(32, \, 44).$ The center of $G$ is $Z(G)= \langle
      i^2 \rangle \simeq
      \mathbb{Z}_2$. We have $[x, \, xk]=[xk, \, z]=1$,
      but $[x, \, z]=i^2$. 
  \end{itemize}
 \end{proof}

\section{Diagonal double Kodaira structures} \label{sec:ddks}
For more details on the material contained in this section, we refer the
reader to \cite{CaPol19} and \cite{Pol20}. Let $G$ be a finite group and
let $b, \, n \geq 2$ be two positive integers.

\begin{definition} \label{def:dks}
  A \emph{diagonal double Kodaira structure} of type $(b, \, n)$ on $G$
  is an ordered set of	$4b+1$ generators
  \begin{equation} \label{eq:dks}
    \S = (\r_{11}, \, \t_{11}, \ldots, \r_{1b}, \, \t_{1b}, \;
      \r_{21}, \,
    \t_{21}, \ldots, \r_{2b}, \, \t_{2b}, \; \z ),
  \end{equation}
  with $o(\z)=n$, such that the following relations are satisfied. We
  systematically use the commutator notation in order to indicate relations of
  conjugacy type, writing for instance $[x, \, y]=zy^{-1}$ instead of
  $xyx^{-1}=z$.
  \begin{itemize}
    \item {Surface relations}
      \begin{align} \label{eq:presentation-0}
	& [\r_{1b}^{-1}, \, \t_{1b}^{-1}] \,
	\t_{1b}^{-1} \, [\r_{1 \,b-1}^{-1}, \,
	\t_{1 \,b-1}^{-1}] \, \t_{1\,b-1}^{-1}
	\cdots [\r_{11}^{-1}, \, \t_{11}^{-1}]
	\, \t_{11}^{-1} \, (\t_{11} \, \t_{12}
	\cdots \t_{1b})=\z \\
	& [\r_{21}^{-1}, \, \t_{21}] \, \t_{21} \,
	[\r_{22}^{-1}, \, \t_{22}] \,
	\t_{22}\cdots	[\r_{2b}^{-1}, \, \t_{2b}]
	\, \t_{2b} \, (\t_{2b}^{-1} \,
	\t_{2 \, b-1}^{-1} \cdots \t_{21}^{-1})
	=\z^{-1}
      \end{align}
    \item {Conjugacy action of} $\r_{1j}$
      \begin{align} \label{eq:presentation-1}
	[\r_{1j}, \, \r_{2k}]& =1  &  \mathrm{if}
	\; \; j < k \\
	[\r_{1j}, \, \r_{2j}]& = 1 & \\
	[\r_{1j}, \, \r_{2k}]& =\z^{-1} \, \r_{2k}\,
	\r_{2j}^{-1} \, \z \, \r_{2j}\,
	\r_{2k}^{-1} \; \;&  \mathrm{if} \;  \;
	j > k \\
	& \\
	[\r_{1j}, \, \t_{2k}]& =1 & \mathrm{if}\;
	\; j < k \\
	[\r_{1j}, \, \t_{2j}]& = \z^{-1} & \\
	[\r_{1j}, \, \t_{2k}]& =[\z^{-1}, \, \t_{2k}]
	& \mathrm{if}\;  \; j > k \\
	& \\
	[\r_{1j}, \,\z]& =[\r_{2j}^{-1}, \,\z] &
      \end{align}
    \item {Conjugacy action of} $\t_{1j}$
      \begin{align} \label{eq:presentation-3}
	[\t_{1j}, \, \r_{2k}]& =1 & \mathrm{if}\;
	\; j < k \\
	[\t_{1j}, \, \r_{2j}]& = \t_{2j}^{-1}\,
	\z\, \t_{2j} & \\
	[\t_{1j}, \, \r_{2k}]& =[\t_{2j}^{-1},\,
	\z] \; \; & \mathrm{if}  \;\; j >
	k \\
	& \\
	[\t_{1j}, \, \t_{2k}]& =1 & \mathrm{if}\;
	\; j < k \\
	[\t_{1j}, \, \t_{2j}]& = [\t_{2j}^{-1}, \,
	\z] & \\
	[\t_{1j}, \, \t_{2k}]& =\t_{2j} ^{-1}\,\z\,
	\t_{2j}\, \z^{-1}\, \t_{2k}\,\z\,
	\t_{2j} ^{-1}\,\z^{-1}\, \t_{2j}\,\t_{2k}^{-1}
	\; \;  & \mathrm{if}\;	\;
	j > k \\
	&  \\
	[\t_{1j}, \,\z]& =[\t_{2j}^{-1}, \,\z] &
      \end{align}
  \end{itemize}
\end{definition}

\begin{remark} \label{inverse-actions}
  From \eqref{eq:presentation-1} and \eqref{eq:presentation-3} we
  can infer
  the corresponding conjugacy actions of $\r_{1j}^{-1}$ and
  $\t_{1j}^{-1}$. We
  leave the cumbersome but standard computations to the reader.
\end{remark}

\begin{remark} \label{rmk:no-abelian-dks}
  Abelian groups admit no diagonal double Kodaira structures. Indeed,
  the
  relation $[\r_{1j}, \, \t_{2j}]=\z^{-1}$  in \eqref{eq:presentation-1}
  provides a non-trivial commutator in $G$, because $o(\z)=n$.
\end{remark}

\begin{remark} \label{rmk:commutators-in-the-center}
  Assume that the nilpotency class of $G$ equals $2$; since  $G$ is non-abelian, this is equivalent to $[G, \, G] \subseteq Z(G)$. Then the relations
  defining a diagonal double Kodaira structure of type $(b, \,
  n)$ assume
  the following simplified form.
  \begin{itemize}
    \item Relations expressing the centrality of $\z$
      \begin{equation} \label{eq:presentation-01-simple}
	[\r_{1j}, \z]=[\t_{1j}, \z]=[\r_{2j},
	\z]=[\t_{2j}, \z]=1
      \end{equation}
    \item {Surface relations}
      \begin{align} \label{eq:presentation-0-simple}
	& [\r_{1b}^{-1}, \, \t_{1b}^{-1}] \,
	[\r_{1 \,b-1}^{-1}, \, \t_{1
	\,b-1}^{-1}] \, \cdots [\r_{11}^{-1}, \,
	\t_{11}^{-1}] \,=\z \\
	& [\r_{21}^{-1}, \, \t_{21}] \,
	[\r_{22}^{-1}, \, \t_{22}] \cdots
	[\r_{2b}^{-1}, \, \t_{2b}] =\z^{-1}
      \end{align}
    \item {Conjugacy action of} $\r_{1j}$
      \begin{align} \label{eq:presentation-1-simple}
	[\r_{1j}, \, \r_{2k}]& =1  &  \mathrm{for \;
	all} \; \; j, \, k \\
	[\r_{1j}, \, \t_{2k}]& = \z^{-\delta_{jk}}
	& \\
      \end{align}
    \item {Conjugacy action of} $\t_{1j}$
      \begin{align} \label{eq:presentation-3-simple}
	[\t_{1j}, \, \r_{2k}]& = \z^{\delta_{jk}} & \\
	[\t_{1j}, \, \t_{2k}]& =1 & \mathrm{for \;
	all} \; \; j, \, k \\
      \end{align}
  \end{itemize}
  where $\delta_{jk}$ stands for the Kronecker symbol. 
\end{remark}

The definition of diagonal double Kodaira structure can be motivated by
means of some well-known concepts in geometric topology. Let $\Sigma_b$
be a closed Riemann surface of genus $b$ and let $\mathscr{P}=(p_1, \,
p_2)$ be an ordered set of two distinct points on it. Let $\Delta \subset
\Sigma_b \times \Sigma_b$ be the diagonal. We denote by
$\mathsf{P}_2(\Sigma_b)$ the \emph{pure braid group} of genus $b$ on two
strands, which is isomorphic to the fundamental group $\pi_1(\Sigma_b \times
\Sigma_b \setminus \Delta, \, \mathscr{P})$. By
Gon\c{c}alves-Guaschi's presentation of surface pure braid groups, see
\cite[Theorem 7]{GG04}, \cite[Theorem 1.7]{CaPol19}, we see that
$\mathsf{P}_2(\Sigma_b)$ can be generated by $4b+1$ elements
\begin{equation} \label{eq:generators-braid}
  \rho_{11}, \, \tau_{11},\ldots, \rho_{1b}, \, \tau_{1b}, \, A_{12}
\end{equation}
subject to the following set of relations.
\begin{itemize}
  \item {Surface relations}
    \begin{align} \label{eq:presentation-0-braids}
      & [\rho_{1b}^{-1}, \, \tau_{1b}^{-1}] \,
      \tau_{1b}^{-1} \, [\rho_{1
	\,b-1}^{-1}, \,
      \tau_{1 \,b-1}^{-1}] \, \tau_{1\,b-1}^{-1} \cdots
      [\rho_{11}^{-1}, \,
      \tau_{11}^{-1}]
      \, \tau_{11}^{-1} \, (\tau_{11} \, \tau_{12} \cdots
      \tau_{1b})=A_{12} \\
      & [\rho_{21}^{-1}, \, \tau_{21}] \, \tau_{21} \,
      [\rho_{22}^{-1}, \,
      \tau_{22}] \,
      \tau_{22}\cdots	[\rho_{2b}^{-1}, \, \tau_{2b}]
      \, \tau_{2b} \,
      (\tau_{2b}^{-1} \,
      \tau_{2 \, b-1}^{-1} \cdots \tau_{21}^{-1})
      =A_{12}^{-1}
    \end{align}
  \item {Conjugacy action of} $\rho_{1j}$
    \begin{align} \label{eq:presentation-1-braids}
      [\rho_{1j}, \, \rho_{2k}]& =1  &	\mathrm{if} \;
      \; j < k \\
      [\rho_{1j}, \, \rho_{2j}]& = 1 & \\
      [\rho_{1j}, \, \rho_{2k}]& =A_{12}^{-1} \, \rho_{2k}\,
      \rho_{2j}^{-1} \,
      A_{12} \, \rho_{2j}\,
      \rho_{2k}^{-1} \; \;&  \mathrm{if} \;  \; j > k \\
      & \\
      [\rho_{1j}, \, \tau_{2k}]& =1 & \mathrm{if}\;  \;
      j < k \\
      [\rho_{1j}, \, \tau_{2j}]& = A_{12}^{-1} & \\
      [\rho_{1j}, \, \tau_{2k}]& =[A_{12}^{-1}, \,
      \tau_{2k}]  & \mathrm{if}\;
      \; j > k \\
      & \\
      [\rho_{1j}, \,A_{12}]& =[\rho_{2j}^{-1}, \,A_{12}] &
    \end{align}
  \item {Conjugacy action of} $\tau_{1j}$
    \begin{align} \label{eq:presentation-3-braids}
      [\tau_{1j}, \, \rho_{2k}]& =1 & \mathrm{if}\;  \;
      j < k \\
      [\tau_{1j}, \, \rho_{2j}]& = \tau_{2j}^{-1}\, A_{12}
      \, \tau_{2j} & \\
      [\tau_{1j}, \, \rho_{2k}]& =[\tau_{2j}^{-1},\,
      A_{12}] \; \; & \mathrm{if}
      \;\; j >
      k \\
      & \\
      [\tau_{1j}, \, \tau_{2k}]& =1 & \mathrm{if}\; \;
      j < k \\
      [\tau_{1j}, \, \tau_{2j}]& = [\tau_{2j}^{-1}, \,
      A_{12}] & \\
      [\tau_{1j}, \, \tau_{2k}]& =\tau_{2j} ^{-1}\,A_{12}
      \, \tau_{2j}\,
      A_{12}^{-1}\, \tau_{2k} \,A_{12} \,
      \tau_{2j} ^{-1}\,A_{12}^{-1}\,
      \tau_{2j}\,\tau_{2k}^{-1} \; \;  &
      \mathrm{if}\;  \;
      j > k \\
      &  \\
      [\tau_{1j}, \,A_{12}]& =[\tau_{2j}^{-1}, \,A_{12}] &
    \end{align}
\end{itemize}
Here the elements $\rho_{ij}$ and $\tau_{ij}$ are the braids depicted in
Figure \ref{fig1}, whereas $A_{12}$ is the braid depicted in Figure
\ref{fig2}.

\begin{figure}[H]
  \begin{center}
    \includegraphics*[totalheight=3 cm]{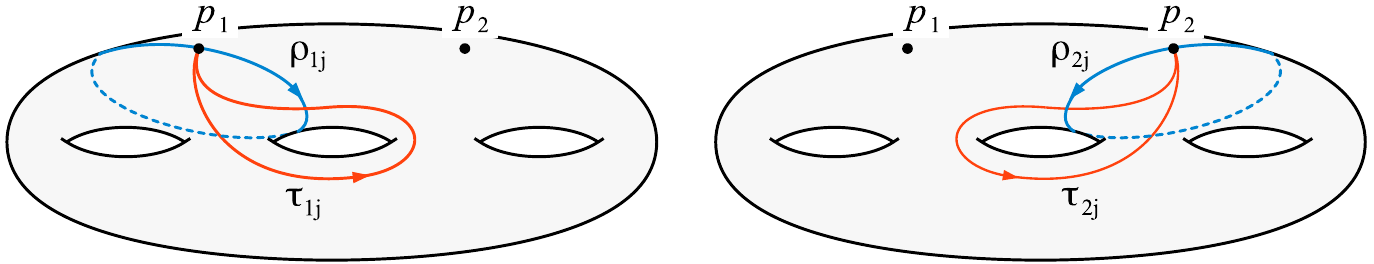}
    \caption{The pure braids $\rho_{1j}$ and $\rho_{2j}$ on
      $\Sigma_b$. If
      $\ell \neq i$, the path corresponding to $\rho_{ij}$
      and $\tau_{ij}$ based
    at $p_{\ell}$ is the constant path.} \label{fig1}
  \end{center}
\end{figure}

\begin{figure}[H]
  \begin{center}
    \includegraphics*[totalheight=1.5 cm]{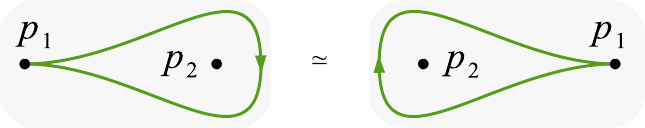}
    \caption{The pure braid $A_{12}$ on $\Sigma_b$} \label{fig2}
  \end{center}
\end{figure}

\begin{remark} \label{rmk:A12}
  Under the identification of $\mathsf{P}_2(\Sigma_b)$ with $\pi_1(\Sigma_b
  \times \Sigma_b \setminus \Delta, \, \mathscr{P})$, the generator $A_{12} \in
  \mathsf{P}(\Sigma_b)$
  represents the homotopy class $\gamma_{\Delta} \in \pi_1(\Sigma_b \times
  \Sigma_b \setminus \Delta, \, \mathscr{P})$ of a loop  in
  $\Sigma_b \times \Sigma_b$ that ``winds once" around the diagonal
  $\Delta$.
\end{remark}

We can now state the following
\begin{proposition} \label{prop:varphi}
  A finite group $G$ admits a diagonal double Kodaira structure  of
  type $(b,
  \,n)$ if and only if there is a surjective group homomorphism
  \begin{equation} \label{eq:varphi}
    \varphi \colon \mathsf{P}_2(\Sigma_b) \to G
  \end{equation}
  such that $\varphi(A_{12})$ has order $n$.
\end{proposition}
\begin{proof}
  If such a $\varphi \colon \mathsf{P}_2(\Sigma_b) \to G$ exists,
  we can obtain
  a diagonal double Kodaira structure on $G$ by setting
  \begin{equation} \label{eq:from-braid-to-structure}
    \r_{ij}=\varphi(\rho_{ij}), \quad \t_{ij}=\varphi(\tau_{ij}),
    \quad
    \z=\varphi({A_{12}}).
  \end{equation}
  Conversely, if $G$ admits a diagonal double Kodaira structure, then
  \eqref{eq:from-braid-to-structure} defines a group homomorphism
  $\varphi
  \colon \mathsf{P}_2(\Sigma_b) \to G$ with the desired properties.
\end{proof}
The braid group $\mathsf{P}_2(\Sigma_b)$ is the middle term of two split
short exact sequences
\begin{equation} \label{eq:oggi}
  1\to \pi_1(\Sigma_b \setminus \{p_i\}, \, p_j) \to \mathsf{P}_2(\Sigma_b) \to
  \pi_1(\Sigma_b, \, p_i) \to 1,
\end{equation}
where $\{i, \, j\} = \{1, \, 2\}$, induced by the two natural projections
of pointed topological spaces
\begin{equation} \label{eq:natural-proj}
  (\Sigma_b \times \Sigma_b \setminus \Delta, \,
  \mathscr{P}) \to (\Sigma_b, \, p_i),
\end{equation}
see \cite[Theorem 1]{GG04}. Since we have
\begin{equation} \label{eq:K1-and-K2}
  \begin{split}
    \pi_1(\Sigma_b \setminus \{p_2\}, \, p_1)&=\langle  \rho_{11}, \,
    \tau_{11}, \ldots,
    \rho_{1b}, \, \tau_{1b}, \; A_{12}	\rangle \\	\pi_1(\Sigma_b
      \setminus \{p_1\}, \,
    p_2)&=\langle  \rho_{21}, \, \tau_{21}, \ldots, \rho_{2b},
    \, \tau_{2b}, \;
    A_{12}	\rangle,
  \end{split}
\end{equation}
it follows that the two subgroups
\begin{equation}
  \begin{split}
    K_1&:=\langle  \r_{11}, \, \t_{11}, \ldots, \r_{1b}, \,
    \t_{1b}, \; \z
    \rangle \\
    K_2&:=\langle  \r_{21}, \, \t_{21}, \ldots, \r_{2b}, \,
    \t_{2b}, \; \z	\rangle
  \end{split}
\end{equation}
are both normal in $G$, and that there are two short exact sequences
\begin{equation} \label{eq:K1K2}
  \begin{split}
    &1 \to K_1 \to G \to Q_2 \to 1 \\
    &1 \to K_2 \to G \to Q_1 \to 1,
  \end{split}
\end{equation}
such the elements $\r_{21}, \, \t_{21}, \ldots, \r_{2b}, \, \t_{2b}$
yield a complete system of coset representatives for $Q_2$, whereas the
elements
$\r_{11}, \, \t_{11}, \ldots, \r_{1b}, \, \t_{1b}$ yield a complete system
of coset
representatives for $Q_1$.

Let us now give a couple of definitions, whose geometrical meaning will
become clear in Section \ref{sec:ddkf}, see in particular Proposition
\ref{prop:invariant-S-G} and Remark \ref{rmk:sdks-characterization}.

\begin{definition} \label{def:signature-ddks}
  Let $\mathfrak{S}$ be a diagonal double Kodaira structure
  of type $(b, \, n)$ on a finite group $G$. Its \emph{signature} is
  defined as
  \begin{equation} \label{eq:signature-ddks}
    \sigma(\mathfrak{S})=\frac{1}{3}\,|G|\,(2b-2)\left(1-\frac{1}{n^2}\right).
  \end{equation}
\end{definition}

\begin{definition} \label{def:strongdks}
  A diagonal double Kodaira structure on $G$ is called
  \emph{strong}
  if $K_1=K_2=G$.
\end{definition}
For later use, let us write down the special case consisting of a diagonal
double Kodaira structure of
type $(2, \, n)$. It is an ordered set of nine generators of $G$
\begin{equation}
  (\r_{11}, \, \t_{11}, \,  \r_{12}, \, \t_{12}, \, \r_{21}, \,
    \t_{21}, \,
  \r_{22}, \, \t_{22}, \, \z ),
\end{equation}
with $o(\z)=n$, subject to the following relations.
\reqnomode
\begin{equation} \label{eq:ddks-genus-2}
  \begin{aligned}
    \mathbf{(S1)} & \,\, [\r_{12}^{-1}, \, \t_{12}^{-1}] \,
    \t_{12}^{-1} \,
    [\r_{11}^{-1}, \, \t_{11}^{-1}] \, \t_{11}^{-1}\, (\t_{11}
    \, \t_{12}) =
    \z & \\ \mathbf{(S2)} & \, \, [\r_{21}^{-1}, \, \t_{21}]
    \; \t_{21} \;
    [\r_{22}^{-1}, \, \t_{22}] \, \t_{22}\, (\t_{22}^{-1} \,
    \t_{21}^{-1})=
    \z^{-1} \\
    & \\
    \mathbf{(R1)} & \, \, [\r_{11}, \, \r_{22}]=1 & \mathbf{(R6)}
    & \, \,
    [\r_{12}, \, \r_{22}]=1 \\
    \mathbf{(R2)} & \, \, [\r_{11}, \, \r_{21}]=1 &
    \mathbf{(R7)} & \, \, [\r_{12}, \, \r_{21}]=
    \z^{-1}\,\r_{21}\,\r_{22}^{-1}\,\z\,\r_{22}\,\r_{21}^{-1} \\
    \mathbf{(R3)} & \, \, [\r_{11}, \, \t_{22}]=1 & \mathbf{(R8)}
    & \, \,
    [\r_{12}, \, \t_{22}]=\z^{-1} \\
    \mathbf{(R4)} & \, \, [\r_{11}, \, \t_{21}]=\z^{-1} &
    \mathbf{(R9)} & \, \,
    [\r_{12}, \, \t_{21}]=[\z^{-1}, \, \t_{21}] \\
    \mathbf{(R5)} & \, \, [\r_{11}, \, \z]=[\r_{21}^{-1}, \,
    \z] & \mathbf{(R10)}
    & \, \, [\r_{12}, \, \z]=[\r_{22}^{-1}, \, \z] \\
    & \\
    \mathbf{(T1)} & \, \, [\t_{11}, \, \r_{22}]=1 & \mathbf{(T6)}
    & \, \,
    [\t_{12}, \, \r_{22}]= \t_{22}^{-1}\, \z \, \t_{22} \\
    \mathbf{(T2)} & \, \, [\t_{11}, \, \r_{21}]=\t_{21}^{-1}\,
    \z \, \t_{21}
    & \mathbf{(T7)} & \, \, [\t_{12}, \, \r_{21}]= [\t_{22}^{-1},
    \, \z] \\
    \mathbf{(T3)} & \, \, [\t_{11}, \, \t_{22}]=1 & \mathbf{(T8)}
    & \, \,
    [\t_{12}, \, \t_{22}]=[\t_{22}^{-1}, \, \z] \\
    \mathbf{(T4)} & \, \, [\t_{11}, \, \t_{21}]=[\t_{21}^{-1},
    \, \z] &
    \mathbf{(T9)} & \, \, [\t_{12}, \, \t_{21}]=
    \t_{22}^{-1}\, \z \,\t_{22} \,\z^{-1} \,\t_{21} \,\z
    \,\t_{22}^{-1}\,\z^{-1}
    \,\t_{22}\,\t_{21}^{-1} \\
    \mathbf{(T5)} & \, \, [\t_{11}, \, \z]=[\t_{21}^{-1}, \,
    \z] & \mathbf{(T10)}
    & \, \, [\t_{12}, \, \z]=[\t_{22}^{-1}, \, \z] \\
  \end{aligned}
\end{equation}
\leqnomode

\begin{remark} \label{rmk::ddks-genus-2-simplified}
  When $[G, \, G] \subseteq Z(G)$, we have
  \begin{equation} \label{z in the center-genus 2}
    \begin{split}
      &[\r_{11},\, \z]=[\t_{11},\, \z]=[\r_{12},\,
      \z]=[\t_{12},\, \z]=1 \\
      &[\r_{21},\, \z]=[\t_{21},\, \z]=[\r_{22},\,
      \z]=[\t_{22},\, \z]=1
    \end{split}
  \end{equation}
  and the previous relations become
  \begin{equation} \label{eq:ddks-genus-2-commutators-in-the-center}
    \begin{aligned}
      \mathbf{(S1')} & \,\, [\r_{12}^{-1}, \, \t_{12}^{-1}]
      \,  [\r_{11}^{-1}, \,
      \t_{11}^{-1}] =  \z & \\ \mathbf{(S2')} & \, \,
      [\r_{21}^{-1}, \, \t_{21}] \;
      [\r_{22}^{-1}, \, \t_{22}]= \z^{-1} \\
      & \\
      \mathbf{(R1')} & \, \, [\r_{11}, \, \r_{22}]=1 &
      \mathbf{(R6')} & \, \,
      [\r_{12}, \, \r_{22}]=1 \\
      \mathbf{(R2')} & \, \, [\r_{11}, \, \r_{21}]=1 &
      \mathbf{(R7')} & \, \,
      [\r_{12}, \, \r_{21}]= 1 \\
      \mathbf{(R3')} & \, \, [\r_{11}, \, \t_{22}]=1 &
      \mathbf{(R8')} & \, \,
      [\r_{12}, \, \t_{22}]=\z^{-1} \\
      \mathbf{(R4')} & \, \, [\r_{11}, \, \t_{21}]=\z^{-1}
      & \mathbf{(R9')} & \, \,
      [\r_{12}, \, \t_{21}]=1 \\ & \\
      \mathbf{(T1')} & \, \, [\t_{11}, \, \r_{22}]=1 &
      \mathbf{(T6')} & \, \,
      [\t_{12}, \, \r_{22}]=	\z  \\
      \mathbf{(T2')} & \, \, [\t_{11}, \, \r_{21}]=\z  &
      \mathbf{(T7')} & \, \,
      [\t_{12}, \, \r_{21}]= 1 \\
      \mathbf{(T3')} & \, \, [\t_{11}, \, \t_{22}]=1 &
      \mathbf{(T8')} & \, \,
      [\t_{12}, \, \t_{22}]=1 \\
      \mathbf{(T4')} & \, \, [\t_{11}, \, \t_{21}]=1 &
      \mathbf{(T9')} & \, \,
      [\t_{12}, \, \t_{21}]=1 \\
    \end{aligned}
  \end{equation}
\end{remark}

\section{Structures on groups of order at most
\texorpdfstring{$32$}{32}}
\label{sec:DDKS}

\subsection{Prestructures} \label{subsec:prestructures}

\begin{definition} \label{def:kodaira-prestructure}
  Let $G$ be a finite group. A \emph{prestructure}
  on $G$ is an ordered set of nine elements
  \begin{equation}
    (\r_{11}, \, \t_{11}, \,  \r_{12}, \, \t_{12}, \, \r_{21},
      \, \t_{21}, \,
    \r_{22}, \, \t_{22}, \, \z ),
  \end{equation}
  with $o(\z)=n \geq 2$, subject to the relations $(\mathrm{R1}), \ldots,
  (\mathrm{R10})$, $(\mathrm{T1}), \ldots, (\mathrm{T10})$ in
  \eqref{eq:ddks-genus-2}.
\end{definition}
In other words, the nine elements must satisfy all the relations
defining a diagonal double Kodaira structure of type $(2, \, n)$, except the
surface relations. In particular, no abelian group admits prestructures. Note
that we are \emph{not} requiring that the
elements of the prestructure generate  $G$.

\begin{proposition} \label{prop:structure-implies-prestructure}
  If a finite group $G$ admits a diagonal double Kodaira structure of
  type $(b, \, n)$, then it admits a prestructure
  with $o(\z)=n$.
\end{proposition}
\begin{proof}
  Consider the ordered set of nine elements
  $(\r_{11}, \, \t_{11}, \,  \r_{12}, \, \t_{12}, \, \r_{21}, \,
    \t_{21},
  \, \r_{22}, \, \t_{22}, \, \z )$
  in Definition \ref{def:dks} and the relations satisfied by them,
  with the exception of the surface relations.
\end{proof}

\begin{remark} \label{remark:noncentral}
  Let $G$ be a finite group that admits a prestructure. Then $\z$ and
  all its conjugates are non-trivial elements of $G$ and so, from relations
  $\mathrm{(R4)}$,
  $\mathrm{(R8)}$, $\mathrm{(T2)}$, $\mathrm{(T6)}$, it
  follows that $\r_{11}, \, \r_{12}, \, \r_{21}, \, \r_{22}$ and
  $\t_{12}, \, \t_{12},
  \,	\t_{21}, \t_{22}$ are non-central elements of $G$.
\end{remark}

\begin{proposition} \label{prop:no-structure-CCT}
  If $G$ is a \emph{CCT}-group, then $G$ admits no prestructures and,
  subsequently, no diagonal double Kodaira structures.
\end{proposition}
\begin{proof}
  The
  second statement
  is a direct consequence of the first one (see Proposition
  \ref{prop:structure-implies-prestructure}), hence
  it suffices then to check
  that $G$ admits no prestructures. Otherwise, keeping in mind Remark
  \ref{remark:noncentral}, we see that $\mathrm{(R6)}$ and $\mathrm{(T1)}$
  imply
  $[\r_{12}, \, \t_{11}]=1$. From this and $\mathrm{(T3)}$
  we get
  $[\r_{12}, \, \t_{22}]=1$, that contradicts $\mathrm{(R8)}$.
\end{proof}

Given a finite group $G$, we define the \emph{socle} of $G$, denoted by
$\mathrm{soc}(G)$, as the intersection of all non-trivial, normal subgroups
of $G$. For instance, $G$ is simple if and only if $\mathrm{soc}(G)=G$.

\begin{definition} \label{def:monolithic-group}
  A finite group $G$ is called \emph{monolithic} if $\mathrm{soc}(G)
  \neq
  \{1\}$. Equivalently, $G$ is monolithic if it contains precisely
  one minimal
  non-trivial, normal subgroup.
\end{definition}

\begin{example} \label{ex:extra-special-is-monolithic}
  If $G$ is an extra-special $p$-group, then $G$ is monolithic and
  $\operatorname{soc}(G)=Z(G)$. Indeed, since $Z(G) \simeq \mathbb{Z}_p$ is
  normal in $G$, by definition of socle we always have $\operatorname{soc}(G)
  \subseteq Z(G)$. On the other hand, every non-trivial, normal subgroup
  of an
  extra-special group contains the center (see \cite[Exercise 9
  p. 146]{Rob96}),
  hence $Z(G) \subseteq \operatorname{soc}(G)$.
\end{example}

\begin{proposition} \label{prop:monolithic-argument} The following holds.
  \begin{itemize}
    \item[$\boldsymbol{(1)}$] Assume that $G$  admits a
      prestructure, whereas
      no proper quotient of $G$ does. Then $G$ is monolithic
      and $\z \in
      \mathrm{soc}(G)$.
    \item[$\boldsymbol{(2)}$] Assume that $G$ admits a
      prestructure, whereas
      no proper subgroup of $G$ does. Then the elements
      of the prestructure
      generate $G$.
  \end{itemize}
\end{proposition}
\begin{proof}
  $\boldsymbol{(1)}$ Let $\mathfrak{S}=(\r_{11}, \, \t_{11}, \,
    \r_{12}, \,
    \t_{12}, \, \r_{21}, \, \t_{21},
  \, \r_{22}, \, \t_{22}, \, \z)$ be a prestructure in $G$.
  Assume that there is a non-trivial
  normal subgroup
  $N$ of $G$ such that $\z \notin N$. Then $\bz \in G/N$ is non-trivial,
  and so
  $\bar{\mathfrak{S}}=(\br_{11}, \, \bt_{11},
    \,	\br_{12},
    \, \bt_{12}, \, \br_{21}, \, \bt_{21}, \, \br_{22}, \, \bt_{22},
  \, \bz )$
  is a prestructure in the quotient group $G/N$, contradiction. Therefore
  we must have $\z \in \mathrm{soc}(G)$, in particular, $G$ is
  monolithic. \smallskip

  $\boldsymbol{(2)}$ Clear, because every prestructure $\mathfrak{S}$
  in $G$
  is also a prestructure in the subgroup $\langle \mathfrak{S} \rangle$.
\end{proof}

\begin{corollary} \label{cor:z-in-p-group}
  Given a prestructure on an extra-special $p$-group $G$, the element $\z$
  is a generator of $Z(G) \simeq \mathbb{Z}_p$.
\end{corollary}
\begin{proof}
  If $G$ is extra-special, every proper quotient of $G$ is abelian,
  hence it admits no prestructures. The result now follows from
  Example \ref{ex:extra-special-is-monolithic} and Proposition
  \ref{prop:monolithic-argument} (1).
\end{proof}

Note that, by Corollary \ref{cor:z-in-p-group}, in the case of extra-special
$p$-groups the choice of calling
$\z$ the element in the prestructure is coherent with presentations
\eqref{eq:H5} and \eqref{eq:G5}. The case of diagonal double Kodaira
structures on extra-special groups of order $32$ will be studied in Subsection
\ref{subsec:DDKS-small32 extra-special}.

\subsection{The case \texorpdfstring{$|G|<32$}{|G|<32}}
\label{subsec:DDKS-small-minor 32}

\begin{proposition} \label{prop:less-32-no-ddks}
  If $|G|<32$, then $G$ admits no diagonal double Kodaira structures.
\end{proposition}

\begin{proof}
  By Corollary \ref{cor:small-CCT}, Proposition \ref{prop:24-no-CCT} and
  Proposition \ref{prop:no-structure-CCT}, it remains only to check that
  the symmetric group $\mathsf{S}_4$ admits no prestructures. We
  start by observing that
  \begin{equation*}
    \mathrm{soc}(\mathsf{S}_4)=\mathsf{V}_4= \langle (1 \, 2)(3\,
    4), \, (1 \,
    3)(2\, 4) \rangle
  \end{equation*}
  and so, by part $\mathrm{(1)}$ of
  Proposition \ref{prop:monolithic-argument},
  if $\mathfrak{S}$ is a prestructure on $\mathsf{S}_4$ then $\z \in
  \mathsf{V}_4$. Let $\x, \, \y\in \mathsf{S}_4$ be such that $[\x, \,
  \y]=\z$. Examining the tables of subgroups of
  $\mathsf{S}_4$ given in \cite{S4}, by straightforward computations we deduce that either
  $\x, \, \y \in C_{\mathsf{S}_4}(\z) \simeq \mathsf{D}_8$ or $\x, \, \y
  \in \mathsf{A}_4$. Every
  pair in $\mathsf{A}_4$ includes at least a $3$-cycle and so, if $[\x, \,
  \y]=\z$ and both $\x$ and $\y$ have even order, then $\x$ and $\y$
  centralize $\z$.

  If $\x \in \mathsf{S}_4$ is a $3$-cycle, then $C_{\mathsf{S}_4}(\x) =
  \langle \x \rangle \simeq \mathbb{Z}_3$. So, from relations
  $(\mathrm{R1})$,
  $(\mathrm{R2})$, $(\mathrm{R3})$, $(\mathrm{R6})$, it follows that,
  if one of
  the elements $\r_{11}, \, \r_{12}, \, \r_{21}, \, \r_{22}, \,
  \t_{22}$ is a
  $3$-cycle, then all these elements generate the same cyclic
  subgroup. This
  contradicts $(\mathrm{R8})$, hence $\r_{11}, \, \r_{12}, \,
  \r_{21}, \,
  \r_{22}, \, \t_{22}$ all have even order.

  Let us look now at relation $(\mathrm{R8})$. Since $\r_{12}, \,
  \t_{22}$ have
  even order, from the previous remark we infer $\r_{12}, \, \t_{22} \in
  C_{\mathsf{S}_4}(\z)$. Let us consider  $\r_{11}$. If $\r_{11}$
  belongs
  to $\mathsf{A}_4$, being an element of even order  it must be
  conjugate to $\z$, and so it commutes with $\z$;
  otherwise, by $(\mathrm{R4})$, both $\r_{11}$
  and $\t_{21}$ commute with $\z$. Summing up, in any case we have
  $\r_{11} \in C_{\mathsf{S}_4}(\z)$.

  Relation $(\mathrm{R5})$ can be rewritten as $\r_{11}\r_{21} \in
  C_{\mathsf{S}_4}(\z)$, hence $\r_{21} \in
  C_{\mathsf{S}_4}(\z)$. Analogously,
  relation $(\mathrm{R10})$ can be rewritten as $\r_{12}\r_{22} \in
  C_{\mathsf{S}_4}(\z)$, hence $\r_{22} \in C_{\mathsf{S}_4}(\z)$.

  Using relation $\mathrm{(R9)}$, we get $\r_{12} \z \in
  C_{\mathsf{S}_4}(\t_{21})$. Since $\r_{12}$ and $\z$ commute and
  their orders are powers of $2$, it follows that $o(\r_{12} \z)$ is also
  a power of $2$. Therefore $\t_{21}$
  cannot be
  a $3$-cycle, otherwise $C_{\mathsf{S}_4}(\t_{21}) \simeq \mathbb{Z}_3$
  and so
  $\r_{12}\z=1$ that, in turn, would imply $[\r_{12}, \, \t_{22}]=1$,
  contradicting
  $(\mathrm{R8})$. It
  follows that $\t_{21}$ has even order and so, since $\r_{11}$
  has even
  order as well, by $\mathrm{(R4)}$ we infer $\t_{21} \in
  C_{\mathsf{S}_4}(z)$.

  Now we can rewrite $(\mathrm{T2})$ as $[\t_{11}, \, \r_{21}]= \z$. If
  $\t_{11}$ were a $3$-cycle, from $(\mathrm{T1})$ we would get $\r_{22}
  \in C_{\mathsf{S}_4}(\t_{11})\simeq \mathbb{Z}_3$, a contradiction
  since
  $\r_{22}$ has even order. Thus $\t_{11}$ has even order and so it
  belongs to
  $C_{\mathsf{S}_4}(z)$, because $\r_{21}$ has even order,
  too. Analogously,
  by using $(\mathrm{T6})$ and	$(\mathrm{T7})$, we infer $\t_{12}\in
  C_{\mathsf{S}_4}(z)$.

  Summarizing, if $\mathfrak{S}$ were a prestructure on $\mathsf{S}_4$
  we
  should have
  \begin{equation}
    \langle \mathfrak{S} \rangle = C_{\mathsf{S}_4}(z) \simeq
    \mathsf{D}_8,
  \end{equation}
  contradicting part $(\mathrm{2})$ of Proposition
  \ref{prop:monolithic-argument}.
\end{proof}

\subsection{The case \texorpdfstring{$|G|=32$}{|G|=32} and
\texorpdfstring{$G$}{G} non-extra-special}
\label{subsec:DDKS-small32 no extra-special}

We start by proving the following partial strengthening of Proposition
\ref{prop:no-structure-CCT}.

\begin{proposition} \label{prop:no-structure-CCT-enhanced}
  Let $G$ be a finite non-abelian group, and let $H$ be the subgroup of $G$
  generated by those elements whose centralizer is non-abelian. If $H$ is
  abelian and $[H : Z(G)] \leq 4$, then $G$ admits no prestructures with
  $\z \in Z(G)$.
\end{proposition}
\begin{proof}
  First of all, remark that $Z(G)$ is a (normal) subgroup of $H$ because $G$
  is non-abelian. Assume now, by contradiction, that the elements $(\r_{11},
    \, \t_{11}, \,
    \r_{12}, \, \t_{12}, \, \r_{21}, \,
  \t_{21}, \, \r_{22}, \, \t_{22}, \, \z)$ form a prestructure on $G$,
  with $\z \in Z(G)$. Then these elements satisfy
  relations $(\mathrm{R1}'), \ldots, (\mathrm{R9}')$, $(\mathrm{T1}'),
  \ldots, (\mathrm{T9}')$ in
  \eqref{eq:ddks-genus-2-commutators-in-the-center}.
  As $H$ is abelian, $(\mathrm{R4}')$ implies that at least one
  between $\r_{11}, \, \t_{21}$ does not belong to $H$.

  Let us assume $\r_{11} \notin H$. Thus $C_G(\r_{11})$ is abelian,
  and so
  $(\mathrm{R2}')$
  and $(\mathrm{R3}')$ yield $[\r_{21}, \, \t_{22}]=1$. From this, using
  $(\mathrm{T2}')$ and $(\mathrm{T3}')$, we infer that $C_G(\t_{22})$ is
  non-abelian. Similar considerations show that $C_G(\r_{21})$ and
  $C_G(\r_{22})$ are non-abelian,  and so we have
  $\r_{21}, \, \r_{22}, \, \t_{22} \in H$. Using $(\mathrm{T2}')$,
  $(\mathrm{T6}')$,
  $(\mathrm{R8}')$, together with the fact that $H$ is abelian, we deduce
  $\t_{11}, \, \t_{12}, \, \r_{12} \notin H$.
  In particular, $C_G(\r_{12})$ is abelian, so $(\mathrm{R7}')$ and
  $(\mathrm{R9}')$ yield $[\r_{21}, \, \t_{21}]=1$; therefore
  $(\mathrm{T2}')$ and $(\mathrm{T4}')$ imply that $C_G(\t_{21})$
  is non-abelian, and so $\t_{21} \in H$.
  Summing up, we have proved that the four elements
  $\r_{21}, \, \t_{21}, \, \r_{22}, \, \t_{22}$
  belong to $H$; since they are all non-central,
  we infer that they yield four non-trivial elements
  in the quotient group $H/Z(G)$. On the other hand,
  we have $[H:Z(G)]\leq 4$, and so $H/Z(G)$ contains
  at most three non-trivial elements; it follows that
  (at least) two among the elements
  $\r_{21}, \, \t_{21}, \, \r_{22}, \, \t_{22}$ have the same image
  in $H/Z(G)$.
  This means that these two elements are of the form $g, \, gz$,
  with $z \in Z(G)$, and so they have the same centralizer.
  But this is impossible: in fact, relations
  \eqref{eq:ddks-genus-2-commutators-in-the-center} show that each
  element in the set $\{\r_{21}, \, \t_{21}, \, \r_{22}, \, \t_{22}\}$
  fails to commute with exactly one element in the set
  $\{\r_{11}, \, \t_{11}, \, \r_{12}, \, \t_{12}\}$, and no
  two elements in $\{\r_{21}, \, \t_{21}, \, \r_{22}, \, \t_{22}\}$
  fail to commute with the same element in
  $\{\r_{11}, \, \t_{11}, \, \r_{12}, \, \t_{12}\}$.

  The remaining case, namely $\t_{21}\notin H$, can be dealt with
  in an analogous way. Indeed, in this situation we obtain $\{\r_{11},
    \, \t_{11},
  \, \r_{12}, \, \t_{12}\}\subseteq H$, that leads to a contradiction
  as before.
\end{proof}

We can now rule out the non-extra-special groups of order $32$.

\begin{proposition} \label{prop:32-no-extra-special-no-structure}
  Let $G$ be a finite group of order $32$ which is not
  extra-special. Then $G$
  admits no diagonal double Kodaira structures.
\end{proposition}
\begin{proof}
  If $G$ is a CCT-group, then the result follows from Proposition
  \ref{prop:no-structure-CCT}. Thus, by Proposition
  \ref{prop:CCT-minore-32},
  we must only consider the cases $G=G(32, \, t)$, where $t
  \in \{6, \, 7, \, 8, \, 43, \, 44\}$. Standard computations using the
  presentations in
  Table \ref{table:32-nonabelian} of Appendix A show that
  all these groups are
  monolithic, and that for all of them $\mathrm{soc}(G)=Z(G)\simeq
  \mathbb{Z}_2$. 
  Since no proper quotients of $G$ admit
  diagonal double Kodaira structures (Proposition \ref{prop:less-32-no-ddks}),
  it follows from Proposition \ref{prop:monolithic-argument} that every
  diagonal double Kodaira structure on $G$ is such that $\z$ is
  the generator of $Z(G)$. Let $H$ be the subgroup of $G$ generated by
  those elements whose centralizer is non-abelian; by Proposition
  \ref{prop:no-structure-CCT-enhanced} we are now
  done, provided that in every case $H$ is abelian and $[H \,: \,Z(G)]\leq
  4$. Let us now show that this is indeed true, leaving the straightforward
  computations to the reader.
  \begin{itemize}
    \item $G=G(32, \, 6).$ In this case
      $\mathrm{soc}(G)=Z(G)=\langle
      x \rangle$ and $H=\langle x, \, y, \, w^2
      \rangle$. Then $H \simeq
      (\mathbb{Z}_2)^3$ and $[H : Z(G)]=4$.
    \item $G=G(32, \, 7).$ In this case
      $\mathrm{soc}(G)=Z(G)=\langle w
      \rangle$ and $H=\langle z, \, u, \, w \rangle$. Then
      $H \simeq \mathbb{Z}_4
      \times \mathbb{Z}_2$ and $[H
      : Z(G)]=4$.
    \item $G=G(32, \, 8).$ In this case
      $\mathrm{soc}(G)=Z(G)=\langle x^4
      \rangle$ and $H=\langle x^2, \, y, \, z^2
      \rangle$. Then $H \simeq
      \mathbb{Z}_4 \times \mathbb{Z}_2$
      and $[H : Z(G)]=4$.
    \item $G=G(32, \, 43).$ In this case
      $\mathrm{soc}(G)=Z(G)=\langle x^4
      \rangle$ and $H=\langle x^2, \, z \rangle$. Then
      $H \simeq \mathbb{Z}_4
      \times \mathbb{Z}_2$ and $[H
      : Z(G)]=4$.
    \item $G=G(32, \, 44).$ In this case
      $\mathrm{soc}(G)=Z(G)=\langle i^2
      \rangle$ and $H=\langle x, \, k \rangle$. Then
      $H \simeq \mathbb{Z}_4
      \times \mathbb{Z}_2$ and $[H
      : Z(G)]=4$.
  \end{itemize}
  This completes the proof.
\end{proof}

\subsection{The case \texorpdfstring{$|G|=32$}{|G|=32} and
\texorpdfstring{$G$}{G} extra-special}
\label{subsec:DDKS-small32 extra-special}

We are now ready to address the case where $|G|=32$ and $G$ is
extra-special. Let us first recall some additional results on extra-special $p$-groups,
referring the reader to \cite{Win72} for more details.

Let $G$ be an extra-special $p$-group of order $p^{2b+1}$ and $\x, \,
\y \in G$. Setting
$(\bar{\x}, \, \bar{\y})=\bar{a}$ where $[\x, \, \y]=\z^a$, the quotient
group $V=G/Z(G) \simeq (\mathbb{Z}_p)^{2b}$ becomes a non-degenerate
symplectic vector space
over $\mathbb{Z}_p$. Looking at $\eqref{eq:H5}$ and $\eqref{eq:G5}$,
we see that in both cases $G=\mathsf{H}_{2b+1}(\mathbb{Z}_p)$ and
$G=\mathsf{G}_{2b+1}(\mathbb{Z}_p)$ we have
\begin{equation} \label{eq:symplectic-form}
  (\bar{\r}_j, \, \bar{\r}_k)=0, \quad (\bar{\t}_j, \, \bar{\t}_k)=0,
  \quad
  (\bar{\r}_j, \, \bar{\t}_k)= -\delta_{jk}
\end{equation}
for all $j, \, k \in \{1, \ldots, b\}$, so that
\begin{equation} \label{eq:symplectic-basis}
  \bar{\r}_1, \, \bar{\t}_1, \ldots, \bar{\r}_b, \, \bar{\t}_b
\end{equation}
is an ordered symplectic basis for $V \simeq (\mathbb{Z}_p)^{2b}$. If
$p=2$, we
can also set $q(\bar{\x})=\bar{c}$, where $\x^2=\z^c$ and $c \in \{0, \,
1\}$; this is a quadratic form on $V$. If $\bar{\x} \in G/Z(G)$ is expressed
in coordinates, with respect to the symplectic basis
\eqref{eq:symplectic-basis}, by the vector $(\xi_1, \, \psi_1, \ldots,
  \xi_b, \,
\psi_b) \in (\mathbb{Z}_2)^{2b}$,  then a straightforward computation yields
\begin{equation} \label{eq:form-of-quadratic-forms}
  q(\bar{\x})=
  \begin{cases}
    \xi_1\psi_1+\cdots+\xi_b \psi_b, & \textrm{if }
    G=\mathsf{H}_{2b+1}(\mathbb{Z}_2) \\
    \xi_1\psi_1+\cdots+\xi_b \psi_b+ \xi_b^2+\psi_b^2 &
    \textrm{if }
    G=\mathsf{G}_{2b+1}(\mathbb{Z}_2).
  \end{cases}
\end{equation}
These are the two possible normal forms for a non-degenerate quadratic form
of dimension $2b$ over $\mathbb{Z}_2$; they have Arf invariant equal to $0$ and $1$, respectively, see for instance \cite{Dye78} or \cite[Chapter 10]{Li97}. In both cases, the symplectic and the quadratic
form are related by
\begin{equation}
  q(\bar{\x} \bar{\y})=q(\bar{\x})+q(\bar{\y})+(\bar{\x}, \, \bar{\y})
  \quad
  \textrm{for all } \bar{\x}, \, \bar{\y} \in V.
\end{equation}
If $\phi \in \mathrm{Aut}(G)$,
then $\phi$ induces a linear map $\bar{\phi} \in \mathrm{End}(V)$; moreover,
if $p=2$, then $\phi$ acts trivially on $Z(G)=[G, \, G] \simeq \mathbb{Z}_2$,
and this in turn
implies that $\phi$ preserves the symplectic form
on $V$. In other words, if we identify $V$ with $(\mathbb{Z}_2)^{2b}$
via the symplectic basis \eqref{eq:symplectic-basis}, we have $\bar{\phi}
\in \mathsf{Sp}(2b, \, \mathbb{Z}_2)$.

We are now in a position to describe the structure of $\mathrm{Aut}(G)$,
see \cite[Theorem 1]{Win72}.

\begin{proposition} \label{prop:Out(G)}
  Let $G$ be an extra-special group of order $2^{2b+1}$. Then the
  kernel of the group homomorphism $\mathrm{Aut}(G) \to \mathsf{Sp}(2b,
  \, \mathbb{Z}_2)$ given by $\phi \mapsto \bar{\phi}$ is the subgroup
  $\mathrm{Inn}(G)$ of inner automorphisms of $G$. Therefore
  $\mathrm{Out}(G)
  = \mathrm{Aut}(G)/\mathrm{Inn}(G)$ embeds in $\mathsf{Sp}(2b,
  \, \mathbb{Z}_2)$. More precisely, $\mathrm{Out}(G)$ coincides
  with the
  orthogonal group $\mathsf{O}_{\epsilon}(2b, \, \mathbb{Z}_2)$,
  of order
  \begin{equation} \label{eq:order-orthogonal}
    |\mathsf{O}_{\epsilon}(2b, \,
    \mathbb{Z}_2)|=2^{b(b-1)+1}(2^b-\epsilon)
    \prod_{i=1}^{b-1}(2^{2i}-1),
  \end{equation}
  associated with the quadratic form
  $\mathrm{\eqref{eq:form-of-quadratic-forms}}$. Here $\epsilon =
  1$ if $G=\mathsf{H}_{2b+1}(\mathbb{Z}_2)$ and $\epsilon = -1$ if
  $G=\mathsf{G}_{2b+1}(\mathbb{Z}_2)$.
\end{proposition}

\begin{corollary} \label{prop:Aut(G)}
  Let $G$ be an extra-special group of order $2^{2b+1}$. We have
  \begin{equation} \label{eq:Aut(G)}
    |\mathrm{Aut}(G)|=2^{b(b+1)+1}(2^b-\epsilon)
    \prod_{i=1}^{b-1}(2^{2i}-1).
  \end{equation}
\end{corollary}
\begin{proof}
  By Proposition \ref{prop:Out(G)} we get
  $|\mathrm{Aut}(G)|=|\mathrm{Inn}(G)|
  \cdot |\mathsf{O}_{\epsilon}(2b, \, \mathbb{Z}_2)|$. Since
  $\mathrm{Inn}(G) \simeq G/Z(G)$ has order $2^{2b}$, the claim
  follows from
  \eqref{eq:order-orthogonal}.
\end{proof}

In particular, plugging $b=2$ in \eqref{eq:Aut(G)},  we can compute the
orders of automorphism groups of extra-special groups of order $32$,
namely
\begin{equation} \label{eq:Order-auto-32}
  |\mathrm{Aut}(\mathsf{H}_5(\mathbb{Z}_2))|=1152, \quad
  |\mathrm{Aut}(\mathsf{G}_5(\mathbb{Z}_2))|=1920.
\end{equation}

Assume now that $\mathfrak{S}=(\r_{11}, \, \t_{11}, \, \r_{12}, \, \t_{12},
\, \r_{21}, \, \t_{21}, \, \r_{22}, \, \t_{22}, \z)$ is a diagonal double
Kodaira
structure of type $(2, \, n)$ on an extra-special group $G$ of order
$32$; by Corollary \ref{cor:z-in-p-group}, the element $\z$ is the generator
of $Z(G) \simeq \mathbb{Z}_2$, hence $n=2$. Then
\begin{equation} \label{eq:bar-Sigma}
  \bar{\mathfrak{S}}=(\br_{11}, \, \bt_{11}, \, \br_{12}, \,
    \bt_{12}, \,
  \br_{21}, \, \bt_{21}, \, \br_{22}, \, \bt_{22})
\end{equation}
is an ordered set of generators for the symplectic
$\mathbb{Z}_2$-vector space $V=G/Z(G) \simeq (\mathbb{Z}_2)^4$, and
\eqref{eq:ddks-genus-2-commutators-in-the-center} yields the relations
\begin{equation} \label{eq:symplectic-relations}
  \begin{split}
    &(\br_{12}, \, \bt_{12})+(\br_{11}, \, \bt_{11})=1, \\
    &(\br_{21}, \, \bt_{21})+(\br_{22}, \, \bt_{22})=1, \\
    &(\br_{1j}, \, \bt_{2k})=\delta_{jk}, \quad (\br_{1j}, \,
    \br_{2k})=0 \\
    &(\bt_{1j}, \, \br_{2k})=\delta_{jk}, \quad (\bt_{1j},
    \, \bt_{2k})=0.
  \end{split}
\end{equation}
Conversely, given any set of generators $\bar{\mathfrak{S}}$ of $V$ as in
\eqref{eq:bar-Sigma}, whose elements satisfy \eqref{eq:symplectic-relations},
a diagonal double Kodaira structure of type
$(b, \, n)=(2, \, 2)$ on $G$ inducing $\bar{\mathfrak{S}}$ is necessarily
of the form
\begin{equation}
  \mathfrak{S}=(\r_{11} \z^{a_{11}}, \, \t_{11} \z^{b_{11}}, \,
    \r_{12}\z^{a_{12}}, \, \t_{12}\z^{b_{12}}, \,
    \r_{21}\z^{a_{21}}, \,
    \t_{21}\z^{b_{21}}, \, \r_{22} \z^{a_{22}}, \, \t_{22}\z^{b_{22}},
  \, \z),
\end{equation}
where $a_{ij}, \, b_{ij} \in \{0, \, 1 \}$. This proves the
following

\begin{lemma} \label{lemma:number-DDKS}
  The total number of diagonal double Kodaira structures of type $(b,
  \, n)=(2,
  \, 2)$
  on an extra-special group $G$ of order $32$ is obtained multiplying
  by $2^8$
  the number of ordered sets of generators $\bar{\mathfrak{S}}$
  of $V$ as in $\eqref{eq:bar-Sigma}$, whose
  elements satisfy \eqref{eq:symplectic-relations}. In particular,
  such a
  number does not depend on $G$.
\end{lemma}

We are now ready to state the main result of this section.
\begin{theorem} \label{thm:main-algebraic}
  A finite group $G$ of order $32$ admits a diagonal double Kodaira
  structure
  if and only if $G$ is extra-special. In this case, the following
  holds.
  \begin{itemize}
    \item[$\boldsymbol{(1)}$] Both extra-special groups of order $32$
      admit $2211840=1152 \cdot
      1920$ distinct
      diagonal double Kodaira structures of type $(b, \, n)=(2, \,
      2)$. Every such a structure $\mathfrak{S}$ is strong
      and satisfies $\sigma(\mathfrak{S})=16$.
    \item[$\boldsymbol{(2)}$] If $G=G(32,
      \,49)=\mathsf{H}_5(\mathbb{Z}_2)$,
      these structures form $1920$ orbits under the action
      of $\mathrm{Aut}(G)$.
    \item[$\boldsymbol{(3)}$] If $G=G(32,
      \,50)=\mathsf{G}_5(\mathbb{Z}_2)$,
      these structures form $1152$ orbits under the action
      of $\mathrm{Aut}(G)$.
  \end{itemize}
\end{theorem}
\begin{proof}
  We already know that non-extra-special	groups of order
  $32$ admit no diagonal double Kodaira structures (Proposition
  \ref{prop:32-no-extra-special-no-structure}) and so, in the sequel, we can
  assume that $G$ is extra-special.

  Looking at the first two relations in \eqref{eq:symplectic-relations},
  we see that we must consider four cases:
  \begin{itemize}
    \item[$\boldsymbol{(a)}$]
      $\; (\br_{12}, \, \bt_{12}) =0, \quad (\br_{11}, \,
      \bt_{11})=1, \quad
      (\br_{21}, \, \bt_{21}) =0, \quad (\br_{22}, \,
      \bt_{22})=1,$
    \item[$\boldsymbol{(b)}$]
      $\; (\br_{12}, \, \bt_{12}) =1, \quad (\br_{11}, \,
      \bt_{11})=0, \quad
      (\br_{21}, \, \bt_{21}) =1, \quad (\br_{22}, \,
      \bt_{22})=0,$
    \item[$\boldsymbol{(c)}$]
      $\; (\br_{12}, \, \bt_{12}) =0, \quad (\br_{11}, \,
      \bt_{11})=1, \quad
      (\br_{21}, \, \bt_{21}) =1, \quad (\br_{22}, \,
      \bt_{22})=0,$
    \item[$\boldsymbol{(d)}$]
      $\; (\br_{12}, \, \bt_{12}) =1, \quad (\br_{11}, \,
      \bt_{11})=0, \quad
      (\br_{21}, \, \bt_{21}) =0, \quad (\br_{22}, \,
      \bt_{22})=1.$
  \end{itemize}

  \medskip
  $\textbf{Case}$ $\boldsymbol{(a)}$. In this case the vectors
  $\br_{11},
  \, \bt_{11}, \, \br_{22}, \, \bt_{22}$ are a symplectic basis of $V$,
  whereas the subspace $W=\langle \br_{12}, \, \bt_{12}, \, \br_{21}, \,
  \bt_{21} \rangle$ is isotropic, namely the symplectic
  form is identically zero on it.  Since $V$ is a symplectic vector space
  of dimension
  $4$, the Witt index of $V$, i.e. the dimension of a maximal isotropic
  subspace of $V$,
  is $\frac{1}{2} \dim(V)=2$, see \cite[Th\'{e}or\`{e}mes 3.10, 3.11]{Ar62}.
  On the other hand, we have $(\br_{12}, \, \bt_{22})=1$ and
  $(\bt_{12}, \, \, \bt_{22})=0$, hence $\br_{12}, \, \bt_{12}$ are linearly
  independent and so they must generate a maximal isotropic subspace;
  it follows that $W=\langle \br_{12}, \, \bt_{12} \rangle$. Let us set now
  \begin{equation}
    \begin{split}
      (\br_{11}, \, \br_{12})&=a, \quad (\br_{11}, \,
      \bt_{12})=b, \quad (\br_{12},
      \, \bt_{11})=c, \quad (\bt_{11}, \, \bt_{12})=d, \\
      (\br_{21}, \, \br_{22})&=e, \quad (\br_{21}, \,
      \bt_{22})=f, \quad (\br_{22},
      \, \bt_{21})=g, \quad (\bt_{21}, \, \bt_{22})=h, \\
    \end{split}
  \end{equation}
  where $a, \, b, \, c, \, d, \, e, \, f, \, g, \, h \in \mathbb{Z}_2$,
  and
  let us express the remaining vectors of $\bar{\mathfrak{S}}$ in
  terms of
  the symplectic basis. Standard computations yield
  \begin{align} \label{eq:expression-symplectic-1}
    \br_{12} &=c\br_{11}+a\bt_{11}+\br_{22}, &\bt_{12}= d \br_{11}
    + b \bt_{11}+
    \bt_{22},\\
    \br_{21} &=\br_{11} + f \br_{22}+e \bt_{22},  &\bt_{21}=
    \bt_{11}+ h \br_{22}+g
    \bt_{22}.
  \end{align}
  Now recall that $W$ is isotropic; then, using the expressions in
  \eqref{eq:expression-symplectic-1} and imposing the relations
  \begin{equation}
    \begin{array}{lll}
      ( \br_{12}, \, \bt_{12})=0,  & \quad ( \br_{12}, \,
      \br_{21})=0, & \quad (
      \br_{12}, \, \bt_{21})=0, \\
      (\br_{21}, \, \bt_{12})=0, & \quad ( \bt_{12}, \,
      \bt_{21})=0, & \quad (
      \br_{21}, \, \bt_{21})=0,
    \end{array}
  \end{equation}
  we get
  \begin{equation}
    \begin{array}{lll}
      ad+bc=1, & \quad a+e=0, & \quad \quad c+g=0, \\
      \quad b+f=0, & \quad d+h=0, & \quad eh+fg=1.
    \end{array}
  \end{equation}
  Summing up, the elements $\br_{12}$, $\bt_{12}$, $\br_{21}$,
  $\bt_{21}$
  can be determined from the symplectic basis via the relations
  \begin{align} \label{eq:expression-symplectic-2}
    \br_{12} &=c\br_{11}+a\bt_{11}+\br_{22}, &\bt_{12}= d \br_{11}
    + b \bt_{11}+
    \bt_{22},\\
    \br_{21} &=\br_{11} + b \br_{22}+a \bt_{22},  &\bt_{21}=
    \bt_{11}+ d \br_{22}+c
    \bt_{22},
  \end{align}
  where $a, \, b, \, c, \, d \in \mathbb{Z}_2$ and
  $ad+bc=1$. Conversely,
  given any symplectic basis $\br_{11}, \, \bt_{11}, \, \br_{22},
  \, \bt_{22}$
  of $V$ and elements $\br_{12}$, $\bt_{12}$, $\br_{21}$, $\bt_{21}$
  as in
  \eqref{eq:expression-symplectic-2}, with $ad+bc=1$, we get a set of
  generators $\bar{\mathfrak{S}}$ of $V$ having the form \eqref{eq:bar-Sigma},
  and whose elements satisfy
  \eqref{eq:symplectic-relations}. Thus, the
  total number
  of such $\bar{\mathfrak{S}}$ in this case
  is given by
  \begin{equation}
    |\mathsf{Sp}(4, \, \mathbb{Z}_2)| \cdot |\mathsf{GL}(2,
    \, \mathbb{Z}_2)|=
    720 \cdot 6 = 4320
  \end{equation}
  and so, by Lemma \ref{lemma:number-DDKS}, the corresponding number
  of diagonal
  double Kodaira structures is $2^8 \cdot 4320 = 1105920$. All
  these structures
  are strong: in fact, we have
  \begin{equation}
    \begin{split}
      K_1 = \langle \r_{11}, \, \t_{11}, \, \r_{12}, \,
      \t_{12} \rangle & = \langle
      \r_{11}, \, \t_{11}, \, \r_{11}^c \t_{11}^a \r_{22},
      \, \r_{11}^d \t_{11}^b
      \t_{22} \rangle \\
      & = \langle \r_{11}, \, \t_{11}, \, \r_{22}, \,
      \t_{22} \rangle =G
    \end{split}
  \end{equation}
  \begin{equation}
    \begin{split}
      K_2 = \langle \r_{21}, \, \t_{21}, \, \r_{22}, \,
      \t_{22} \rangle & = \langle
      \r_{11}\r_{22}^b\t_{22}^a, \,
      \t_{11}\r_{22}^d\t_{22}^c, \, \r_{22}, \,
      \t_{22}  \rangle \\
      & = \langle \r_{11}, \, \t_{11}, \, \r_{22}, \,
      \t_{22} \rangle =G,
    \end{split}
  \end{equation}
  the last equality following in both cases because $\langle
  \br_{11}, \,
  \bt_{11}, \, \br_{22}, \, \bt_{22} \rangle =V$ and $[\r_{11},
  \, \t_{11}]=\z$.

  \medskip
  $\textbf{Case}$ $\boldsymbol{(b)}$. In this situation, the elements
  $\{\br_{12}, \, \bt_{12}, \, \br_{21}, \, \bt_{21} \}$ form a symplectic
  basis for $V$, whereas $W=\langle \br_{11}, \, \bt_{11}, \, \br_{22}, \,
  \bt_{22} \rangle$ is an isotropic subspace. The same calculations as in
  case $(a)$ show that there are again $1105920$ diagonal double Kodaira
  structures.

  \medskip
  $\textbf{Case}$ $\boldsymbol{(c)}$. This case do not
  occur. In fact, in
  this situation
  the subspace $W= \langle \br_{12}, \, \bt_{12}, \, \br_{21}
  \rangle$ is
  isotropic. Take a linear combination of its generators giving the zero
  vector, namely
  \begin{equation} \label{eq:linear-combination}
    a\br_{12}+b\bt_{12}+c \br_{21}=0.
  \end{equation}
  Pairing
  with $\bt_{21}$,
  $\bt_{22}$, $\br_{22}$, we get $c=a=b=0$. Thus,
  $\br_{12}, \, \bt_{12}, \, \br_{21}$
  are linearly independent, and $W$ is an isotropic subspace
  of dimension
  $3$ inside the $4$-dimensional symplectic space $V$, contradiction.

  \medskip
  $\textbf{Case}$ $\boldsymbol{(d)}$. This case is obtained from $(c)$
  by exchanging the indices $1$ and $2$, so it does not occur, either.

  \medskip

  Summarizing, we have found $1105920$ diagonal double Kodaira
  structures in
  cases $(a)$ and $(b)$, and no structure at all in cases $(c)$ and
  $(d)$. So
  the total number of diagonal double Kodaira structures on $G$
  is  $2211840$,
  and this concludes the proof of part $\boldsymbol{(1)}$.
  \medskip

  Now observe that, since every diagonal double Kodaira structure
  $\mathfrak{S}$
  generates $G$, the only automorphism $\phi$ of $G$ fixing $\mathfrak{S}$
  elementwise is the identity. This means that $\mathrm{Aut}(G)$
  acts
  freely on the set of diagonal double Kodaira structures, hence
  the number
  of orbits is obtained dividing $2211840$ by
  $|\mathrm{Aut}(G)|$. Parts $\boldsymbol{(2)}$ and $\boldsymbol{(3)}$
  now
  follow from \eqref{eq:Order-auto-32}, and we are done.
\end{proof}

\begin{example} \label{example:explicit DDKS}
  Let us give an explicit example of diagonal double Kodaira
  structure on
  an extra-special group $G$ of order $32$, by using the construction
  described in the proof of part $\boldsymbol{(1)}$ of Theorem
  \ref{thm:main-algebraic}. Referring to the presentations
  for $\mathsf{H}_5(\mathbb{Z}_2)$ and $\mathsf{G}_5(\mathbb{Z}_2)$
  given in
  Proposition \ref{prop:extra-special-groups}, we start by choosing
  in both
  cases the following elements, whose images give a symplectic basis
  for $V$:
  \begin{equation}
    \r_{11}=\r_1, \quad \t_{11}=\t_1, \quad \r_{22}=\r_2,
    \quad \t_{22}=\t_2.
  \end{equation}
  Choosing $a=d=1$ and $b=c=0$ in \eqref{eq:expression-symplectic-2},
  we find
  the remaining elements, obtaining  the diagonal double Kodaira
  structure
  \begin{equation}
    \begin{aligned}
      \r_{11}&=\r_1, & \quad \t_{11}& =\t_1, & \quad
      \r_{12}& =\r_2\, \t_1,
      & \quad \t_{12}&=\r_1 \, \t_2 \\ \r_{21}& =\r_1\,\t_2,
      & \quad \t_{21} & =\r_2\,
      \t_1,
      & \quad \r_{22}& =\r_2, &\quad \t_{22}&=\t_2.
    \end{aligned}
  \end{equation}
\end{example}

\begin{remark} \label{rmk:comparison-minimal-order}
  Theorem \ref{thm:main-algebraic} should be compared with  previous
  results of
  \cite{CaPol19} and \cite{Pol20}, regarding the construction of
  diagonal
  double Kodaira structures
  on some extra-special groups of order at least $2^7=128$. We emphasize that the examples on extra-special groups of order $32$ presented here
  are really
  new, in the sense that they cannot be obtained by taking the image of
  structures
  on extra-special groups of bigger order: in fact, an extra-special
  group
  admits no non-abelian proper quotients, cf. Example
  \ref{ex:extra-special-is-monolithic}.
\end{remark}

Let us end this section with the restatement of Theorem \ref{thm:main-algebraic}
in terms of admissible epimorphisms from surface braid groups to finite groups.

\begin{corollary} \label{cor:quotient-of-P2}
Let $G$ be a finite group admitting an admissible epimorphism $\varphi
  \colon \mathsf{P}_2(\Sigma_b) \to G$. Then
  $|G| \geq 32$, with equality if and only if $G$ is
  extra-special. Moreover,
  the following holds.
  \begin{itemize}
    \item[$\boldsymbol{(1)}$] For both extra-special groups
      $G$ of order $32$, there are $2211840=1152 \cdot 1920$
      admissible epimorphisms $\varphi \colon \mathsf{P}_2(\Sigma_2)
      \to G$. For all of them, $\varphi(A_{12})$ is the
      generator of $Z(G)$, so $n=2$.
    \item[$\boldsymbol{(2)}$] If $G=G(32, \,
      49)=\mathsf{H}_5(\mathbb{Z}_2)$,
      these epimorhisms
      form $1920$ orbits under the natural action of
      $\operatorname{Aut}(G)$.
    \item[$\boldsymbol{(3)}$] If $G=G(32, \, 50)=
      \mathsf{G}_5(\mathbb{Z}_2)$,
      these epimorhisms
      form $1152$ orbits under the natural action of
      $\operatorname{Aut}(G)$.
  \end{itemize}
\end{corollary}


\section{Geometrical application: diagonal double Kodaira fibrations}
\label{sec:ddkf}



The aim of this section is to show how the existence of  diagonal double
Kodaira structures is equivalent to the existence of some special
double Kodaira fibrations (see the Introduction for the definition), that we call \emph{diagonal double Kodaira fibrations}. We closely follow the treatment given in \cite[Section 4]{Pol20}.

With a slight abuse of notation, in the sequel we will use the symbol
$\Sigma_b$ to indicate both a smooth complex curve of genus $b$ and its
underlying real surface. By Grauert-Remmert's extension theorem and Serre's
GAGA, the group epimorphism $\varphi \colon \mathsf{P}_2(\Sigma_b) \to G$
described in Proposition \ref{prop:varphi} yields the existence of a smooth,
complex, projective surface $S$ endowed with a Galois cover
\begin{equation}
  \mathbf{f} \colon S \to \Sigma_b \times \Sigma_b,
\end{equation}
with Galois group $G$ and branched precisely over $\Delta$ with branching
order $n$, see \cite[Proposition 3.4]{CaPol19}. Composing the left
homomorphisms
in \eqref{eq:oggi} with $\varphi \colon \mathsf{P}_2(\Sigma_b) \to G$,
we get two homomorphisms
\begin{equation} \label{eq:varphi-i}
  \varphi_1 \colon \pi_1(\Sigma_b -\{p_2\}, \, p_1) \to G, \quad
  \varphi_2
  \colon \pi_1(\Sigma_b -\{p_1\}, \, p_2) \to G,
\end{equation}
whose respective images coincide with the subgroups $K_1$ and $K_2$ defined
in \eqref{eq:K1K2}. By construction, these are
the homomorphisms induced by the restrictions $\mathbf{f}_i \colon \Gamma_i
\to \Sigma_b$  of the Galois cover $\mathbf{f} \colon S \to \Sigma_b
\times \Sigma_b$ to the fibres of the two natural projections $\pi_i
\colon \Sigma_b \times \Sigma_b \to \Sigma_b$. Since $\Delta$ intersects
transversally at a single point all the fibres of the natural projections,
it follows that both such restrictions are branched at precisely one point,
and the number of connected components of the smooth curve $\Gamma_i \subset
S$ equals the index $m_i:=[G : K_i]$ of $K_i$ in $G$.

So, taking the Stein factorizations of the compositions $\pi_i \circ
\mathbf{f} \colon S \to \Sigma_b$ as in the diagram below
\begin{equation} \label{dia:Stein-Kodaira-gi}
  \begin{tikzcd}
    S \ar{rr}{\pi_i \circ \mathbf{f}}  \ar{dr}{f_i} & &
    \Sigma_b	\\
    & \Sigma_{b_i} \ar{ur}{\theta_i} &
  \end{tikzcd}
\end{equation}
we obtain two distinct Kodaira fibrations $f_i \colon S \to \Sigma_{b_i}$,
hence a double Kodaira fibration by considering the product morphism
\begin{equation}
  f=f_1 \times f_2 \colon S \to \Sigma_{b_1} \times \Sigma_{b_2}.
\end{equation}
\begin{definition} \label{def:diagonal-double-kodaira-fibration}
  We call $f \colon S \to \Sigma_{b_1} \times \Sigma_{b_2}$ the
  \emph{diagonal
  double Kodaira fibration} associated with the diagonal double Kodaira
  structure $\S$ on the finite group $G$. Conversely, we will say that a
  double Kodaira fibration $f \colon S \to \Sigma_{b_1} \times
  \Sigma_{b_2}$
  is \emph{of diagonal type} $(b, \, n)$ if there exists a finite
  group $G$
  and a diagonal double Kodaira structure $\S$ of type $(b, \, n)$
  on it such
  that $f$ is associated with $\S$.
\end{definition}

\medskip

Since the morphism $\theta_i \colon \Sigma_{b_i} \to \Sigma_b$ is \'{e}tale
of degree $m_i$, by using the Hurwitz formula we obtain
\begin{equation} \label{eq:expression-gi}
  b_1 -1 =m_1(b-1), \quad  b_2 -1 =m_2(b-1).
\end{equation}
Moreover, the fibre genera $g_1$, $g_2$ of the Kodaira fibrations $f_1
\colon S \to \Sigma_{b_1}$, $f_2 \colon S \to \Sigma_{b_2}$ are computed
by the formulae
\begin{equation} \label{eq:expression-gFi}
  2g_1-2 = \frac{|G|}{m_1} (2b-2 + \mathfrak{n} ), \quad 2g_2-2 =
  \frac{|G|}{m_2} \left( 2b-2 + \mathfrak{n} \right),
\end{equation}
where $\mathfrak{n}:= 1 - 1/n$. Finally, the surface $S$ fits into a diagram
\begin{equation} \label{dia:degree-f-general}
  \begin{tikzcd}
    S \ar{rr}{\mathbf{f}}  \ar{dr}{f} & & \Sigma_b \times
    \Sigma_b  \\
    & \Sigma_{b_1} \times \Sigma_{b_2} \ar[ur, "\theta_1 \times
    \theta_2"{sloped, anchor=south}] &
  \end{tikzcd}
\end{equation}
so that the diagonal double Kodaira fibration $f \colon S \to  \Sigma_{b_1}
\times \Sigma_{b_2}$ is a finite cover of degree $\frac{|G|}{m_1m_2}$,
branched precisely over the curve
\begin{equation} \label{eq:branching-f}
  (\theta_1 \times \theta_2)^{-1}(\Delta)=\Sigma_{b_1} \times_{\Sigma_b}
  \Sigma_{b_2}.
\end{equation}
Such a curve is always smooth, being the preimage of a smooth divisor via an
\'{e}tale morphism. However, it is reducible in general, see \cite[Proposition
3.11]{CaPol19}. The invariants of $S$ can be now computed as follows,
see \cite[Proposition 3.8]{CaPol19}.

\begin{proposition} \label{prop:invariant-S-G}
  Let $f \colon S \to \Sigma_{b_1} \times \Sigma_{b_2}$ be a diagonal
  double
  Kodaira fibration, associated with a diagonal double Kodaira
  structure $\S$
  of type $(b, \, n)$ on a finite group $G$. Then we have
  \begin{equation} \label{eq:invariants-S-G}
    \begin{split}
      c_1^2(S) & = |G|\,(2b-2) ( 4b-4 + 4 \mathfrak{n}
      - \mathfrak{n}^2 ) \\
      c_2(S) & =   |G|\,(2b-2) (2b-2 + \mathfrak{n})
    \end{split}
  \end{equation}
  where $\mathfrak{n}=1-1/n$.	As a consequence, the slope and the signature
  of $S$ can be
  expressed as
  \begin{equation} \label{eq:slope-signature-S-G}
    \begin{split}
      \nu(S) & = \frac{c_1^2(S)}{c_2(S)} = 2+ \frac{2
	\mathfrak{n} -
      \mathfrak{n}^2}{2b-2 + \mathfrak{n} } \\
      \sigma(S)& = \frac{1}{3}\left(c_1^2(S) - 2 c_2(S)
      \right)
      =\frac{1}{3}\,|G|\,(2b-2)\left(1-\frac{1}{n^2}\right)
      =\sigma(\mathfrak{S})
    \end{split}
  \end{equation}
\end{proposition}

\begin{remark} \label{rmk:sdks-characterization}
  By definition, the diagonal double Kodaira structure $\S$ is strong
  if and only if $m_1=m_2=1$, that in turn implies $b_1=b_2=b$,
  i.e.,
  $f=\mathbf{f}$. In other words, $\S$ is strong if and only
  if no Stein
  factorization as in \eqref{dia:Stein-Kodaira-gi} is needed or,
  equivalently,
  if and only if the Galois cover $\mathbf{f}\colon S \to \Sigma_b
  \times
  \Sigma_b$ induced by \eqref{eq:varphi} is already a double Kodaira
  fibration,
  branched on the diagonal $\Delta \subset \Sigma_b \times \Sigma_b$.
\end{remark}

\begin{remark} \label{rmk:rokhlin}
  Every Kodaira fibred surface $S$ satisfies $\sigma(S) >0$, see the
  introduction to \cite{LLR17}; moreover, since $S$ is a differentiable
  $4$-manifold that is a real surface bundle, its signature is
  divisible by
  $4$, see \cite{Mey73}. In addition, if $S$ is associated with
  a diagonal
  double Kodaira structure of type $(b, \, n)$, with $n$ odd, then $K_S$
  is $2$-divisible in $\textrm{Pic}(S)$ and so $\sigma(S)$ is a positive
  multiple of $16$ by Rokhlin's theorem, see \cite[Remark 3.9]{CaPol19}.
\end{remark}

We are now ready to give a geometric restatement of the algebraic results of Section \ref{sec:DDKS} in terms of double Kodaira fibrations.

\begin{theorem} \label{thm:main-geometric}
  Let $G$ be a finite group and
  \begin{equation} \label{eq:G-cover-2}
    \mathbf{f} \colon S \to \Sigma_b \times \Sigma_b
  \end{equation}
  be a Galois cover with Galois group $G$,  branched over the diagonal
  $\Delta$ with branching order $n \geq 2$. Then the following hold.
  \begin{itemize}
    \item[$\boldsymbol{(1)}$] We have $|G| \geq 32$, with
      equality precisely
      when $G$	is extra-special.
    \item[$\boldsymbol{(2)}$] If $G=G(32, \,
      49)=\mathsf{H}_5(\mathbb{Z}_2)$ and $b=2$,
      there are $1920$ $G$-covers of type
      \textrm{\eqref{eq:G-cover-2}}, up to
      cover isomorphisms.
    \item[$\boldsymbol{(3)}$] If $G=G(32, \,
      50)=\mathsf{G}_5(\mathbb{Z}_2)$ and $b=2$,
      there are $1152$ $G$-covers of type
      \textrm{\eqref{eq:G-cover-2}}, up to
      cover isomorphisms.
  \end{itemize}
  Finally, in both cases $\boldsymbol{(2)}$ and $\boldsymbol{(3)}$, we have
  $n=2$ and each
  cover $\mathbf{f}$ is a double Kodaira fibration with
  \begin{equation}
    b_1=b_2=2, \quad g_1=g_2=41,  \quad
    \sigma(S)=16.
  \end{equation}
\end{theorem}
\begin{proof}
  A cover  as in
  \eqref{eq:G-cover-2},
  branched over $\Delta$ with order $n$, exists if and only if $G$
  admits a
  double Kodaira structure of type $(b, \, n)$; additionally, the number of
  such covers,
  up to cover isomorphisms, equals the number of structures up
  the natural
  action of $\mathrm{Aut}(G)$. Then,  $\boldsymbol{(1)}$,
  $\boldsymbol{(2)}$
  and $\boldsymbol{(3)}$ can be deduced from the corresponding
  statements in
  Theorem \ref{thm:main-algebraic}. The same theorem tells us that
  all double
  Kodaira structures on an extra-special group of order $32$ are
  strong, hence the cover $\mathbf{f}$ is already a double Kodaira
  fibration and no Stein factorization is needed (Remark
  \ref{rmk:sdks-characterization}).
  The fibre genera, the slope and the
  signature of $S$ can be now computed by using
  \eqref{eq:expression-gFi} and \eqref{eq:slope-signature-S-G}.
\end{proof}
As a consequence, we obtain a sharp lower  bound for the signature of a
diagonal double Kodaira fibration or, equivalently, of a diagonal double
Kodaira structure.
\begin{corollary} \label{cor:bound-signature}
  Let $f \colon S \to \Sigma_{b_1} \times \Sigma_{b_2}$ be a diagonal
  double
  Kodaira fibration, associated with a diagonal double Kodaira structure
  of type
  $(b, \, n)$ on a finite group $G$. Then $\sigma(S) \geq 16$,
  and equality
  holds precisely when $(b, \, n)=(2, \, 2)$ and $G$ is an extra-special
  group of order $32$.
\end{corollary}
\begin{proof}
  Theorem \ref{thm:main-algebraic} implies $|G| \geq 32$. Since $b
  \geq 2$
  and $n \geq 2$, from \eqref{eq:slope-signature-S-G} we get
  \begin{equation}
    \sigma(S) =\frac{1}{3}\,|G|\,(2b-2)\left(1-\frac{1}{\;
    n^2}\right) \\
    \geq \frac{1}{3} \cdot 32 \cdot (2\cdot 2 -2)
    \left(1-\frac{1}{\; 2^2}\right)
    = 16,
  \end{equation}
  and equality holds if and only if we are in the situation described in
  Theorem \ref{thm:main-geometric}, namely, $b=n=2$  and $G$ an
  extra-special
  group of order $32$.
\end{proof}

These results provide, in particular, new ``double solutions'' to a problem,
posed by G. Mess, from Kirby's problem list in low-dimensional topology
\cite[Problem 2.18 A]{Kir97}, asking what is the smallest number $b$ for
which there exists a real surface bundle over a real surface with base genus
$b$ and non-zero signature. We actually have $b=2$, also for double Kodaira
fibrations, as shown in \cite[Proposition 3.19]{CaPol19} and \cite{Pol20} by using double Kodaira structures of type $(2, \, 3)$ on
extra-special groups of order $3^5$. Those fibrations had signature $144$
and fibre genera $325$; by using our new examples, we can now
substantially lower both these values.

\begin{theorem} \label{thm:new-Kirby}
  Let $S$ be the diagonal double Kodaira surface associated
  with a strong diagonal double Kodaira structure of type $(b, \, n)=(2, \,
  2)$ on an extra-special group $G$ of order $32$. Then the real manifold $M$
  underlying
  $S$ is a closed, orientable $4$-manifold of signature $16$ that
  can be realized as a real surface bundle over a real surface of
  genus $2$,
  with fibre  genus $41$, in two different ways.
\end{theorem}

Theorem \ref{thm:main-geometric} also implies  the following partial answer
to \cite[Question 3.20]{CaPol19}.

\begin{corollary} \label{cor:min-genus-signature}
  Let $g_{\mathrm{min}}$ and $\sigma_{\mathrm{min}}$ be the minimal
  possible
  fibre genus and signature for a double Kodaira fibration $f \colon
  S \to
  \Sigma_2 \times \Sigma_2$. Then we have
  \begin{equation}
    g_{\mathrm{min}} \leq 41, \quad \sigma_{\mathrm{min}} \leq 16.
  \end{equation}
\end{corollary}

In fact, it is an interesting question whether $16$ and $41$ are the minimum
possible values
for the signature and the fibre genus of a (non necessarily diagonal)
double Kodaira fibration
$f \colon S \to \Sigma_2 \times \Sigma_2$, but we will not address this
problem here.

\begin{remark} \label{rmk:comparison-with-loenne}
  Constructing (double) Kodaira fibrations with small signature is
  a rather difficult problem. As far as we know, before our work the only
  examples with signature $16$ were the one described in \cite[Theorem 1.1]{BD02} and the ones listed in \cite[Table 3, Cases
      6.2, 6.6,
    6.7 (Type 1), 6.9]{LLR17}. The examples provided by Theorem
  \ref{thm:main-geometric}
  are new, since both the base genera and the fibre genera are different. Note
  that our results also show that \emph{every} curve of genus $2$
  (and not only some special curve with extra automorphisms)
  is the base of a double Kodaira fibration with signature $16$. Thus,
  we obtain two families
  of dimension $3$ of such fibrations that, to the best of our knowledge,
  provide the first examples of a positive-dimensional families
  of double Kodaira fibrations with small signature.
\end{remark}

\section*{Acknowledgements}
F. Polizzi was partially supported by GNSAGA-INdAM. He thanks Andrea Causin
for drawing the figures and Z\"{o}nke Rollenske for answering some of his
questions via e-mail. Both authors thank the anonymous referee for constructive criticism and suggestions that helped to improve the presentation of these results. They are also grateful
to Ian Agol, Yves de Cornulier, ``Jonathan", Derek Holt, Max Horn, Moishe
Kohan, Roberto Pignatelli, ``Primoz", Geoff Robinson, John Shareshian,
Remy van Dobben de Bruyn, Will Sawin  for their precious answers and comments
in the MathOverflow threads \\ \\
\noindent \url{https://mathoverflow.net/questions/357453} \\
\url{https://mathoverflow.net/questions/366044} \\
\url{https://mathoverflow.net/questions/366771} \\
\url{https://mathoverflow.net/questions/368628} \\
\url{https://mathoverflow.net/questions/371181} \\
\url{https://mathoverflow.net/questions/379272} \\
\url{https://mathoverflow.net/questions/380292} \\
\url{https://mathoverflow.net/questions/390447}

\newpage

\section*{Appendix. Non abelian groups of order \texorpdfstring{$24$}{24}
  and \texorpdfstring{$32$}{32}} \label{Appendix_A}

\begin{table}[H]
  \begin{center}
    \begin{tabularx}{\textwidth}{@{}ccc@{}}
      \hline
      $\mathrm{IdSmallGroup}(G)$ & $G$ &
      $\mathrm{Presentation}$ \\
      \hline \hline
      $G(24, \, 1)$ & $\mathsf{D}_{8, \, 3, \, -1}$ &
      $\langle x, \, y \; | \;
      x^8=y^3=1, \, xyx^{-1}=y^{-1} \rangle$	\\
      \hline
      $G(24, \, 3)$ & $\mathsf{SL}(2, \, \mathbb{F}_3)$
      & $\langle x, \, y, \,
      z \; | \;  \, x^3=y^3=z^2=xyz \rangle$ \\
      \hline
      $G(24, \, 4)$ & $\mathsf{Q}_{24}$ & $\langle x, \,
      y, \, z \; | \;  \,
      x^6=y^2=z^2=xyz \rangle$ \\
      \hline
      $G(24, \, 5)$ & $\mathsf{D}_{2, \, 12, \, 5}$ &
      $\langle x, \, y \; | \;
      x^2=y^{12}=1, \, xyx^{-1}=y^5 \rangle$ \\
      \hline
      $G(24, \, 6)$ & $\mathsf{D}_{24}$  &  $\langle x,
      \, y \; | \; x^2=y^{12}=1,
      \, xyx^{-1}=y^{-1} \rangle$ \\
      \hline
      $G(24, \, 7)$ & $\mathbb{Z}_2 \times \mathsf{D}_{4,
      \, 3, \, -1}$
      &  $\langle z \; | \; z^2 =1 \rangle \times \langle
      x, \, y \; | \; x^4=y^3=1,
      \, xyx^{-1}=y^{-1} \rangle$\\
      \hline
      $G(24, \, 8)$ & $((\mathbb{Z}_2)^2 \times \mathbb{Z}_3
      )\rtimes \mathbb{Z}_2$
      & $\langle x,\, \,  y,\, z, \, w \; | \;
      x^2=y^2=z^2=w^3=1,$ \\
      $$ & $$
      & $[y,z]=[y,w]=[z,w]=1,$
      \\
      $$ & $$ & $xyx^{-1}=y,\; xzx^{-1}=zy, \;
      xwx^{-1}=w^{-1} \rangle$ \\
      \hline
      $G(24, \, 10)$ & $\mathbb{Z}_3 \times \mathsf{D}_8$
      & $\langle z \; | \;
      z^3=1 \rangle \times \langle x, \, y \; | \;  x^2=y^4=1, \,
      xyx^{-1}=y^{-1} \rangle$	 \\
      \hline
      $G(24, \, 11)$ & $\mathbb{Z}_3 \times \mathsf{Q}_8$
      & $\langle z \; | \;
      z^3=1 \rangle \times \langle i, \, j, \, k \; | \;
      i^2=j^2=k^2=ijk \rangle$ \\
      \hline
      $G(24, \, 12)$ & $\mathsf{S}_4 $ & $\langle x, \,
      y \; | \; x=(12), \,
      y=(1234) \rangle$  \\
      \hline
      $G(24, \, 13)$ & $\mathbb{Z}_2 \times \mathsf{A}_4 $
      & $\langle z \; | \;
      z^2 =1 \rangle \times \langle x, \, y \; | \;
      x=(12)(34), y=(123) \rangle$ \\
      \hline
      $G(24, \, 14)$ & $(\mathbb{Z}_2)^2 \times
      \mathsf{S}_3$ & $\langle z, \,
      w \; | \; z^2=w^2=[z, \, w]=1 \rangle$ \\
      & & $\times \langle x, \, y \; | \; x=(12), \,
      y=(123) \rangle$ \\
      \hline
    \end{tabularx}
  \end{center}
  \caption[caption]{Nonabelian groups of
    order $24$.\\ \hspace{\textwidth}Source:
  \url{groupprops.subwiki.org/wiki/Groups_of_order_24}}
  \label{table:24-nonabelian}
\end{table}

\begin{table}[H]
  \begin{center}
    \begin{tabularx}{\textwidth}{@{}ccc@{}}
      \hline
      $\mathrm{IdSmallGroup}(G)$ & $G$ &
      $\mathrm{Presentation}$ \\
      \hline \hline
      $G(32, \, 2)$ & $(\mathbb{Z}_4 \times \mathbb{Z}_2)
      \rtimes \mathbb{Z}_4$
      & $\langle x, \, y, \, z \; | \; x^4=y^4=z^2=1,$	\\
      & & $[x, \, y]=z, \, [x, \, z]=[y, \, z]=1 \rangle$\\
      \hline
      $G(32, \, 4)$ & $\mathsf{D}_{4, \, 8, \, 5}$ &
      $\langle x, \, y \; | \;
      x^4=y^8=1, xyx^{-1}=y^5 \rangle$ \\
      \hline
      $G(32, \, 5)$ &  $(\mathbb{Z}_8 \times	\mathbb{Z}_2)
      \rtimes \mathbb{Z}_2$ & $\langle x, \, y, \, z \;|\;
      x^8=y^2=z^2=1, \,
      [x, \, y]=1, $\\
      & & $zxz^{-1}=x^5y, \, zyz^{-1}=y \rangle$  \\
      \hline
      $G(32, \, 6)$ & $(\mathbb{Z}_2)^3  \rtimes
      \mathbb{Z}_4 $ & $\langle x, \,
      y, \, z, \, w \mid x^2 = y^2 = z^2 = w^4 = 1,$ \\
      & &
      $[x, \, y]=1, \, [x, \, z]=1, \, [y,\, z]=1, $ \\
      & & $wxw^{-1} = x, \, wyw^{-1} = xy, \, wzw^{-1}
      = yz \rangle$  \\
      \hline
      $G(32, \, 7)$ & $(\mathbb{Z}_8 \rtimes \mathbb{Z}_2)
      \rtimes \mathbb{Z}_2$
      &  $\langle x, \, y, \, z, \, u, \, w \mid
      y^2=z^2=w^2=1, $\\
      & & $u^2=w^{-1}, \, x^2=u, \, (yz)^2=1, \,
      (yu^{-1})^2=1,$ \\
      & & $uzu^{-1}=z^{-1}, \, xyzx^{-1}=y^{-1} \rangle$ \\
      \hline
      $G(32, \, 8)$ & $(\mathbb{Z}_2)^2  \,. \,(\mathbb{Z}_4
      \times \mathbb{Z}_2)  $ & $\langle x, \, y, \,
      z \mid x^8=y^2=1, \,
      z^2=x^4,$ \\
      & & $xy=yx^5, \, [y, \, z]=1, \, xz=zxy^{-1}
      \rangle$ \\
      \hline
      $G(32, \, 9)$ & $(\mathbb{Z}_8 \times  \mathbb{Z}_2)
      \rtimes \mathbb{Z}_2$ & $\langle x, \, y, \, z \;|\;
      x^8=y^2=z^2=1, \,
      [x, \, y]=1, $\\
      & & $zxz^{-1}=x^3y, \, zyz^{-1}=y \rangle$  \\
      \hline
      $G(32, \, 10)$ & $\mathsf{Q}_8	 \rtimes
      \mathbb{Z}_4$ & $\langle i, \, j,
      \, k, \, x \mid
      i^2=j^2=k^2=ijk, \, x^4=1,$ \\
      & & $xix^{-1}=j, \, xjx^{-1}=i, \, xkx^{-1}=k^{-1}
      \rangle$ \\
      \hline
      $G(32, \, 11)$ & $(\mathbb{Z}_4)^2  \rtimes
      \mathbb{Z}_2
      $ &  $\langle x, \, y, \, z \mid x^4=y^4=[x, \,
      y]=1, \, z^2=1,$	\\
      & & $zxz^{-1}=y, \, zyz^{-1}=x \rangle$ \\
      \hline
      $G(32, \, 12)$ & $\mathsf{D}_{8, \, 4, \, 3}$ &
      $\langle x, \, y \; | \;
      x^8=y^4=1, xyx^{-1}=y^3 \rangle$	\\
      \hline
      $G(32, \, 13)$ & $\mathsf{D}_{4, \, 8, \, 3}$ &
      $\langle x, \, y \; | \;
      x^4=y^8=1, xyx^{-1}=y^3 \rangle$	\\
      \hline
      $G(32, \, 14)$ & $\mathsf{D}_{4, \, 8, \, -1}$ &
      $\langle x, \, y \; | \;
      x^4=y^8=1, xyx^{-1}=y^{-1} \rangle$ \\
      \hline
      $G(32, \, 15)$ & $\mathbb{Z}_4	 \, . \,
      \mathsf{D}_8 $ & $\langle x, \,
      y, \, z, \, u, \, w \mid w^2=1, \, z^2=u^2=w^{-1},
      $ \\
      & & $x^2=u, \, y^2=z, \, xzx^{-1}=z^{-1}, $ \\
      & &  $[y, \, u]=1, \, xyxu=y^{-1} \rangle$ \\
      \hline
      $G(32, \, 17)$ & $\mathsf{D}_{2, \, 16, \, 9}$&
      $\langle x, \, y \mid
      x^2=y^{16}=1, \, xyx^{-1}=y^9 \rangle$ \\
      \hline
      $G(32, \, 18)$ & $\mathsf{D}_{32} $ & $\langle x,
      \, y \mid x^2=y^{16}=1,
      xyx^{-1}=y^{-1} \rangle$ \\
      \hline
      $G(32, \, 19)$ & $\mathsf{QD}_{32} $ & $\langle x,
      \, y \mid x^2=y^{16}=1,
      xyx^{-1}=y^7 \rangle$ \\
      \hline
      $G(32, \, 20)$ & $\mathsf{Q}_{32} $ & $ \langle x,
      \, y, \, z \mid
      x^8=y^2=z^2=xyz \rangle$\\
      \hline
      $\color{black}{G(32, \, 22)}$ &
      $\color{black}{\mathbb{Z}_2 \times
	((\mathbb{Z}_4
      \times \mathbb{Z}_2) \rtimes \mathbb{Z}_2) }$ &
      $\langle w \mid w^2=1 \rangle
      \times $\\
      & & $\langle x, \, y, \, z \mid x^4=y^2=z^2=1, \,
      [x, \, y]=1, $ \\
      & & $zxz^{-1}=xy, \, zyz^{-1}=y \rangle$ \\
      \hline
      $G(32, \, 23)$ & $\mathbb{Z}_2 \times \mathsf{D}_{4,
      \, 4, \, 3}$ &
      $\langle z \mid z^2=1 \rangle \times \langle x, \,
      y \mid x^4=y^4=1, \,
      xyx^{-1}=y^3 \rangle$ \\
      \hline
      $\color{black}{G(32, \, 24)}$ &
      $\color{black}{(\mathbb{Z}_4)^2
      \rtimes \mathbb{Z}_2 }$ & $\langle x, \, y, \,
      z \mid x^4=y^4=z^2=1,  $	\\
      & & $[x, \, y]=1, \, zxz^{-1}=x, \, zyz^{-1}=x^2y
      \rangle$ \\
      \hline
      $\color{black}{G(32, \, 25)}$ &
      $\color{black}{\mathbb{Z}_4 \times
      \mathsf{D}_8}$
      & $\langle z \mid z^4=1 \rangle \times \langle x,
      \, y \mid x^2=y^4=1, \,
      xyx^{-1}=y^{-1} \rangle$ \\
      \hline
      $\color{black}{G(32, \, 26)}$ &
      $\color{black}{\mathbb{Z}_4 \times
      \mathsf{Q}_8}
      $ & $\langle z \mid z^4=1 \rangle \times \langle i,
      \, j, \, k \mid
      i^2=j^2=k^2=ijk  \rangle$ \\
      \hline
      $\color{black}{G(32, \, 27)}$ &
      $\color{black}{(\mathbb{Z}_2)^3  \rtimes
      (\mathbb{Z}_2)^2}
      $ & $\langle x, \, y, \, z, \, a, \, b \mid $ \\
      & & $x^2=y^2=z^2=a^2=b^2=1, $ \\
      & & $[x, \, y]=[y, \, z]=[x, \, z]=[a, \, b]=1,$ \\
      & & $axa^{-1}=x, \, aya^{-1}=y, \, aza^{-1}=xz, $\\
      & & $bxb^{-1}=x, \, byb^{-1}=y, \, bzb^{-1}=yz
      \rangle$\\
      \hline
      $G(32, \, 28)$ & $(\mathbb{Z}_4 \times
      (\mathbb{Z}_2)^2) \rtimes \mathbb{Z}_2$
      & $\langle x, \, y, \, z, \, w \mid x^4=y^2=z^2=w^2=1,
      $  \\
      & & $[x, y]=[x, \, z]=[y, \, z]=1, $\\
      & & $wxw^{-1}=x^{-1}, \, wyw^{-1}=z, \, wzw^{-1}=y
      \rangle$ \\
      \hline
      $G(32, \, 29)$ & $(\mathbb{Z}_2 \times \mathsf{Q}_8)
      \rtimes  \mathbb{Z}_2 $
      &  $\langle x, \, i, \, j, \, k, \, z \mid x^2=z^2=1,
      \, i^2=j^2=k^2=ijk,$  \\
      & & $[x, \, i]=[x, \, j]=[x, \, k]=1, $ \\
      & & $zxz^{-1}=x, \, ziz^{-1}=i, \, zjz^{-1}=xj^{-1}
      \rangle$	\\
      \hline
      $G(32, \, 30)$ & $(\mathbb{Z}_4 \times
      (\mathbb{Z}_2)^2) \rtimes \mathbb{Z}_2$
      & $\langle x, \, y, \, z, \, w \mid x^4=y^2=z^2=w^2=1,
      $  \\
      & & $[x, y]=[x, \, z]=[y, \, z]=1, $\\
      & & $wxw^{-1}=xy, \, wyw^{-1}=y, \, wzw^{-1}=x^2z
      \rangle$ \\
      \hline
    \end{tabularx}
  \end{center}
\end{table}

\newpage

\begin{table}[H]
  \begin{center}
    \begin{tabularx}{\textwidth}{@{}ccc@{}}
      \hline
      $\mathrm{IdSmallGroup}(G)$ & $G$ &
      $\mathrm{Presentation}$ \\
      \hline \hline
      $G(32, \, 31)$ & $(\mathbb{Z}_4)^2  \rtimes
      \mathbb{Z}_2
      $ &  $\langle x, \, y, \, z \mid x^4=y^4=[x, \,
      y]=1, \, z^2=1,$	\\
      & & $zxz^{-1}=xy^2, \, zyz^{-1}=x^2y \rangle$ \\
      \hline
      $G(32, \, 32)$ & $(\mathbb{Z}_2)^2
      \, . \,  (\mathbb{Z}_2)^3$
      & $\langle x, \, y, \, z, \, u, \, w \mid u^2=w^2=1,
      $    \\
      & & $u=z^2, \, u=x^{-2}, \, w=y^{-2}, $\\
      & & $yxy^{-1}=x^{-1}, \, [y, \, z]=1, \, xzxwz=1
      \rangle$ \\
      \hline
      $G(32, \, 33)$ &	$(\mathbb{Z}_4)^2  \rtimes
      \mathbb{Z}_2
      $ &  $\langle x, \, y, \, z \mid x^4=y^4=[x, \,
      y]=1, \, z^2=1,$	\\
      & & $zxz^{-1}=xy^2, \, zyz^{-1}=x^2y^{-1} \rangle$ \\
      \hline
      $G(32, \, 34)$ &	$(\mathbb{Z}_4)^2  \rtimes
      \mathbb{Z}_2
      $ &  $\langle x, \, y, \, z \mid x^4=y^4=[x, \,
      y]=1, \, z^2=1,$	\\
      & & $zxz^{-1}=x^{-1}, \, zyz^{-1}=y^{-1} \rangle$ \\
      \hline
      $G(32, \, 35)$ &	$\mathbb{Z}_4 \rtimes \mathsf{Q}_8$
      & $\langle x, \, i,
      \, j, \, k \mid
      x^4=1, \, i^2=j^2=k^2=ijk,$ \\
      & & $ixi^{-1}=x^{-1}, \, jxj^{-1}=x^{-1}, \,
      kxk^{-1}=x \rangle$ \\
      \hline
      $G(32, \, 37)$ & $(\mathbb{Z}_8 \times	\mathbb{Z}_2)
      \rtimes \mathbb{Z}_2$ & $\langle x, \, y, \, z \;|\;
      x^8=y^2=z^2=1, \,
      [x, \, y]=1, $\\
      & & $zxz^{-1}=x^5, \, zyz^{-1}=y \rangle$  \\
      \hline
      $G(32, \, 38)$ &	$(\mathbb{Z}_8 \times  \mathbb{Z}_2)
      \rtimes \mathbb{Z}_2$ & $\langle x, \, y, \, z \;|\;
      x^8=y^2=z^2=1, \,
      [x, \, y]=1, $\\
      & & $zxz^{-1}=x, \, zyz^{-1}=x^4y \rangle$  \\
      \hline
      $G(32, \, 39)$ & $\mathbb{Z}_2 \times \mathsf{D}_{16}$
      & $\langle z \mid z^2=1
      \rangle \times \langle x, \, y \mid x^2=y^8=1,
      xyx^{-1}=y^{-1} \rangle$	\\
      \hline
      $G(32, \, 40)$ &	$\mathbb{Z}_2 \times
      \mathsf{QD}_{16}$ & $\langle z \mid
      z^2=1 \rangle \times \langle x, \, y \mid x^2=y^8=1,
      xyx^{-1}=y^3 \rangle$  \\
      \hline
      $G(32, \, 41)$ & $\mathbb{Z}_2 \times \mathsf{Q}_{16}
      $ &  $\langle w \mid
      w^2=1 \rangle \times \langle x, \, y, \, z \mid
      x^4=y^2=z^2=xyz \rangle$\\
      \hline
      $G(32, \, 42)$ & $(\mathbb{Z}_8 \times	\mathbb{Z}_2)
      \rtimes \mathbb{Z}_2$ & $\langle x, \, y, \, z \;|\;
      x^8=y^2=z^2=1, \,
      [x, \, y]=1, $\\
      & & $zxz^{-1}=x^3, \, zyz^{-1}=x^4y \rangle$  \\
      \hline
      $G(32, \, 43)$ & $\mathbb{Z}_8 \rtimes
      (\mathbb{Z}_2)^2$ & $\langle x, \,
      y, \, z \mid x^8=1, \, y^2=z^2=[y, \, z]=1,$\\
      & & $yxy^{-1}=x^{-1}, \, zxz^{-1}=x^5 \rangle$ \\
      \hline
      $G(32, \, 44)$ & $(\mathbb{Z}_2 \times \mathsf{Q}_8)
      \rtimes  \mathbb{Z}_2 $
      &  $\langle x, \, i, \, j, \, k, \, z \mid x^2=z^2=1,
      \, i^2=j^2=k^2=ijk,$  \\
      & & $[x, \, i]=[x, \, j]=[x, \, k]=1, $ \\
      & & $zxz^{-1}=xi^2, \, ziz^{-1}=j, \, zjz^{-1}=i
      \rangle$ \\
      \hline
      $G(32, \, 46)$ & $(\mathbb{Z}_2)^2
      \times \mathsf{D}_8$ & $\langle z, \, w \mid
      z^2=w^2=[z, \, w]=1 \rangle$ \\
      &  & $\times \langle x, \, y \mid x^2=y^4=1, \,
      xyx^{-1}=y^{-1} \rangle$	 \\
      \hline
      $G(32, \, 47)$ & $(\mathbb{Z}_2)^2
      \times \mathsf{Q}_8$ & $\langle z, \, w \mid
      z^2=w^2=[z, \, w]=1 \rangle$ \\
      &  & $\times \langle i, \, j, \, k \mid
      i^2=j^2=k^2=ijk  \rangle$  \\
      \hline
      $G(32, \, 48)$ & $(\mathbb{Z}_4 \times
      (\mathbb{Z}_2)^2) \rtimes \mathbb{Z}_2$
      & $\langle x, \, y, \, z, \, w \mid x^4=y^2=z^2=w^2=1,
      $  \\
      & & $[x, y]=[x, \, z]=[y, \, z]=1, $\\
      & & $wxw^{-1}=x, \, wyw^{-1}=y, \, wzw^{-1}=x^2z
      \rangle$ \\
      \hline
      $G(32, \, 49)$ & $\mathsf{H}_5(\mathbb{Z}_2)$ &
      $\langle	\mathsf{r}_1,
      \, \mathsf{t}_1,
      \,\mathsf{r}_2,\, \mathsf{t}_2, \, \z \mid
      \mathsf{r}_{j}^2 =
      \mathsf{t}_{j}^2=\mathsf{z}^2=1,$ \\
      & & $[\mathsf{r}_{j}, \, \mathsf{z}]  =
      [\mathsf{t}_{j}, \, \mathsf{z}]= 1,$ \\
      & & $[\mathsf{r}_j, \, \mathsf{r}_k]= [\mathsf{t}_j,
      \, \mathsf{t}_k] = 1,$ \\
      & & $[\mathsf{r}_{j}, \,\mathsf{t}_{k}] =\mathsf{z}^{-
      \delta_{jk}} \,
      \rangle,$ \quad see \eqref{eq:H5} \\
      \hline
      $G(32, \, 50)$ & $\mathsf{G}_5(\mathbb{Z}_2)$ &
      $\langle \,  \mathsf{r}_1,
      \, \mathsf{t}_1,
      \, \mathsf{r}_2,\, \mathsf{t}_2, \, \z \mid
      ,\mathsf{r}_{1}^2 =
      \mathsf{t}_{1}^2=\mathsf{z}^2=1,$ \\
      & &  $\mathsf{r}_{2}^2 =\mathsf{t}_{2}^2=\mathsf{z},$
      \\
      & & $[\mathsf{r}_{j}, \, \mathsf{z}]  =
      [\mathsf{t}_{j}, \, \mathsf{z}]= 1,$ \\
      & & $[\mathsf{r}_j, \, \mathsf{r}_k]= [\mathsf{t}_j,
      \, \mathsf{t}_k] = 1,$ \\
      & & $[\mathsf{r}_{j}, \,\mathsf{t}_{k}] =\mathsf{z}^{-
      \delta_{jk}} \,
      \rangle,$ \quad see \eqref{eq:G5} \\
      \hline
    \end{tabularx}
  \end{center}
  \caption[caption]{Nonabelian groups of order $32$.\\
    \hspace{\textwidth}Source:
  \url{groupprops.subwiki.org/wiki/Groups_of_order_32}}
  \label{table:32-nonabelian}
\end{table}

\section*{Research Data Policy and Data Availability Statements}
Data sharing not applicable to this article as no datasets were generated or analysed during the current study.

\end{document}